\documentclass{daj}
\usepackage{epigraph}
\usepackage[centertags]{amsmath}
\usepackage{amssymb}
\usepackage{amsthm}
\usepackage{ifthen}
\usepackage{enumerate}
\usepackage{xcolor}

\newcommand{\vectorSequence}[1][n]{
\ifthenelse{\equal{#1}{0}}{\underline{\eta_{#1}}}{}
\ifthenelse{\equal{#1}{n}}{\underline{\eta_{n}}}{}
\ifthenelse{\equal{#1}{\underline{d}}}{\underline{\eta_{n,\textbf{d}}}}{}
}

\newcommand{\C}{\mathbb{C}}
\newcommand{\F}{\mathbb{F}}
\newcommand{\N}{\mathbb{N}}
\newcommand{\PP}{\mathbb{P}}

\newcommand{\Z}{\mathbb{Z}}

\newcommand{\e}{\mathrm{e}}

\newcommand{\Mod}[1]{\ (\text{mod}\ #1)}
\renewcommand{\Mod}[1]{{\ifmmode\text{\rm\ (mod~$#1$)}\else\discretionary{}{}{\hbox{ }}\rm(mod~$#1$)\fi}}

\newtheorem{theorem}{Theorem}
\newtheorem{lemma}{Lemma}

\dajAUTHORdetails{%
  title = {Exponential sums with  reducible polynomials}, 
  author = {C\'ecile Dartyge and Greg Martin},
  plaintextauthor = {Cecile Dartyge and Greg Martin},
  keywords = {exponential sums, equidistribution of polynomial roots},
}   

\dajEDITORdetails{%
   year={2019},
   number={15},
   received={11 July 2018},   
   revised={19 April 2019},    
   published={13 November 2019},  
   doi={10.19086/da.10793},       
}   

\begin{document}

\begin{frontmatter}[classification=text]


\author[cecile]{C\'ecile Dartyge}
\author[greg]{Greg Martin\thanks{Supported in part by a Natural Sciences and Engineering Research Council of Canada Discovery Grant.}}

\begin{abstract}
Hooley proved that if $f\in \Z [X]$ is irreducible of degree $\ge 2$, then the fractions $\{ r/n\}$, $0<r<n$ with $f(r)\equiv 0\Mod n$,
 are uniformly distributed in $]0,1[$. In this paper we study such problems for reducible polynomials of degree~$2$ and~$3$ and for finite products of linear factors. In particular, we establish asymptotic formulas for exponential sums over these normalized roots.
\end{abstract}
\end{frontmatter}


 \section{Introduction}
Let $f(X)$ be a polynomial of degree at least $2$ with integer coefficients, and let $h$ be a nonzero integer.
We consider for $x\ge1$ the exponential sums
\begin{equation}\label{defS}
S(f,x,h)=\sum_{n\le x}\sum_{\substack{{r\Mod n}\\{f(r)\equiv 0\Mod n}}}\e \Big (\frac{ hr}n\Big )
\end{equation}
with the standard notation $\e (t)=\exp (2i\pi t)$; our interest is in fixed~$f$ and~$h$ while $x$ tends to infinity.
In the case $h=1$ we will write simply $S(f,x)$.
Hooley~\cite[Theorem 1]{H64} proved that if $f$ is irreducible, then $S(f,x,h)=o(x)$ when $x\rightarrow\infty$.
 By Weyl's criterion, this implies that the fractions $r/n$, where $0<r<n$ and $f(r)\equiv 0\Mod n$, are uniformly distributed in $]0,1[$.
The condition that $\deg f\ge2$ is necessary in this result, as the roots of a linear polynomial are not equidistributed in $]0,1[$. For example, if $f(n)=n+b$, the exponential term in equation~\eqref{defS} is $\e (-hb/n)=1+O(1/n)$ for fixed~$h$ and~$b$, and then $S(f,x,h)= x+O(\log x)$. 
In general, when $f(n)=an+b$ one can similarly obtain a formula of the form $S(f,x,h)=xC(f,h)+o(x)$ for some constant $C(f,h)$.

When $f(n)=n^2-D$ with $D$ not a square, Hooley~\cite{H63} obtained (using a different method) the more precise
bound $S(f,x,h)\ll_{f,h} x^{3/4}\log x$. The exponent $3/4$ in this result was improved to $2/3$ via the theory of automorphic forms by Hejhal \cite{He78}, Bykovski \cite{B84}, and Zavorotny~\cite{Z84}.
A variant of this problem is the distribution of the fractions $r/p$, where $0<r<p$ and $f(r)\equiv 0\Mod p$, for primes~$p$. Duke, Friedlander, and Iwaniec~\cite{DFI95} proved that for $f(X)=aX^2+2bX+c$ with $ac-b^2>0$, we have
 \begin{equation*}
\sum_{p\le x}\sum_{\substack{{0\le r<p}\\{f(r)\equiv 0\Mod p}}}\e\Big (\frac{hr}{p}\Big )=o(\pi (x)).
 \end{equation*}
 
In all of the above results, the polynomial $f$ is assumed to be irreducible; what can we say about the sums $S(f,x,h)$ when $f$ is reducible? In this paper we examine reducible polynomials of degree~$2$ and $3$ (in which case $f$ has at least one linear factor) and on polynomials that factor completely into linear factors. The hope is to obtain, not just an upper bound, but an actual asymptotic formula for $S(f,x,h)$, analogously to the situation described above where $f$ is itself linear.

For example, Sitar and the second author~\cite{MS11} obtained a general bound for reducible quadratic polynomials~$f$
with discriminant~$D$:
\begin{equation}\label{MS}
S(f,x,h)\ll \sqrt{D}\prod_{p\mid h}\Big ( 1+\frac{7}{\sqrt{p}}\Big )x(\log x)^{\sqrt{2}-1}(\log\log x)^{5/2}.
\end{equation}
Since the number of summands of $S(f,x,h)$ has order $x\log x$ in this case, a consequence is that the the fractions $r/n$, where $0<r<n$ and $f(r)\equiv 0\Mod n$, are uniformly distributed in $]0,1[$ for reducible quadratic polynomials. However, this bound is still large enough to disguise a potential asymptotic main term of size~$x$ caused by the roots of the linear factors of~$f$.

In our first theorem, which is proved in Section~\ref{1+1 sec}, we provide such an asymptotic formula for exponential sums
with reducible quadratics:

\begin{theorem}\label{1+1}
Let $a,b,c,d$ be fixed integers with $ac\not =0$, $(a,b)=(c,d)=1$, and $ad\ne bc$, and set $f(n)=(an+b)(cn+d)$.
Then for any nonzero integer $h$, there exists $C(f,h)\in\C$ such that for any $\varepsilon>0$, we have
\begin{equation}\label{result}
S(f,x,h)=xC(f,h)+O(x^{4/5+\varepsilon}),
\end{equation}
where the implicit constant depends on $f$, $h$, and~$\varepsilon$.
When $(h,ad-bc)=1$, the error term in equation~\eqref{result} can be improved to $O(x^{3/4+\varepsilon})$.
\end{theorem}

\noindent
The proof of this theorem provides an explicit but complicated formula for $C(f,h)$ (see equation~\eqref{C1 general} below).
In the particular case $h=1$, we obtain $S(f,x)=xC(f,1)+O(x^{4/5+\varepsilon})$ with
\begin{equation*}
C(f,1)=\Big (\frac{\mu (a)}{a}+\frac{\mu (c)}{c}\Big )\frac{6}{\pi ^2}\prod_{p\mid ac}\Big ( 1-\frac{1}{p^2}\Big )^{-1}.
\end{equation*}
(Note in particular that it is possible for $C(f,1)$ to equal $0$, namely if neither $a$ nor~$c$ is squarefree; in such a case, Theorem~\ref{1+1} is technically an upper bound rather than an asymptotic formula. The analogous pedantic comment applies to Theorem~\ref{k-linear} below.)

Our second result handles the case $f(n)=n(n^2+1)$.
In Section~\ref{1+2 sec} we will prove:

\begin{theorem}\label{1+2}
For $f(n)=n(n^2+1)$, we have
\begin{equation*}
S(f,x)=\frac{3x}4 \prod_{p\equiv1\Mod4} \Big (1-\frac2{p^2}\Big )+O\Big( \frac{x(\log\log x)^{7/2}}{(\log x)^{1-1/\sqrt{2}}} \Big).
\end{equation*}
\end{theorem}

\noindent
It is also possible to generalize Theorem~\ref{1+2} to $S(f,x,h)$ for a general nonzero integer $h$, as in the work of Hooley, though we do not do so here.
It is likely possible to extend Theorem \ref{1+2} to general products of two polynomials $f_1f_2$ with $f_1$ linear and $f_2$ irreducible quadratic
by adapting some ideas of \cite{LO12} or \cite{BBDT12}. However, such a
 generalization with a general $f_2$ in place of $n^2+1$ is not straightforward; the case where $f_2$ has positive discriminant
seems more difficult.

The next result, which we prove in Sections~\ref{1+1+1 sec} and~\ref{ses1 sec}, handles in detail a special product of three linear polynomials:

\begin{theorem}\label{1+1+1}
For $f(n)=n(n+1)(2n+1)$, we have 
\begin{equation*}
S(f,x)=x\prod_{p\ge 3}\Big (1-\frac{2}{p^2}\Big )+O\Big (\frac{x(\log\log x )^6}{(\log x)^2}\Big ).
\end{equation*}
\end{theorem}

It was clear that this result could be generalized to arbitrary products of three linear polynomials. However, after reading an earlier version of this manuscript, the anonymous referee pointed out that the method can be extended to give an asymptotic formula for $S(f,x,h)$ when $f$ is the product of four linear polynomials, as well as a nontrivial upper bound for $S(f,x,h)$ when $f$ is a product of any number of linear polynomials. We prove the following result in Sections~\ref{ses2 sec} and~\ref{k sec}:


\begin{theorem}\label{k-linear}
Let $h$ be a nonzero integer, let $k\ge3$ be an integer, and let $f(n)=\prod_{i=1}^k(a_in+b_i)$ where the coefficients $a_1,b_1,\ldots,a_k, b_k$ are integers satisfying
$(a_j,b_j)=1$ for all $1\le i\le k$ and $a_ib_j\ne a_jb_i$ for all $1\le i<j\le k$.
\begin{enumerate}[(i)]
\item If $k=3$ or $k=4$, then there exists $C(f,h)\in\C$ such that
\begin{equation*}
S(f,x,h)=C(f,h)x+O\big( x(\log x)^{k-5}(\log\log x)^6 \big).
\end{equation*}
\item
If $k\ge 5$, then $S(f,x,h)\ll x(\log x)^{k-5}(\log\log x)^6$.
\end{enumerate}
In both cases, the implicit constants depend on~$f$ and~$h$.
\end{theorem}

As the number of fractions $r/n$, where $0<r<n$ and $f(r)\equiv 0\Mod n$, for all $n\le x$ has order of magnitude $x(\log x)^{\kappa-1}$, where $\kappa$ is the number of distinct irreducible factors of~$f$, these theorems all imply that the appropriate sequences of fractions are uniformly distributed in $]0,1[$.

It is also worth remarking that it is only for convenience that we hypothesize that the linear factors in Theorems~\ref{1+1} and~\ref{k-linear} are not proportional to one another. Indeed, more generally consider a polynomial $f = g^2 h$: 
the roots of $f$ modulo squarefree integers $n$ are unaffected by the repeated factor $g^2$, while if $p^2\mid n$ then the roots of $g^2\Mod n$ form arithmetic progressions with common difference $n/p$, and thus their contribution to the exponential sum $S(f,x,h)$ vanishes completely.

Before proceeding to the proofs of our theorems, we establish in the next section a lemma that is used repeatedly throughout the paper. We adopt the following notation and conventions throughout this paper: when $a$ and $b$ are relatively prime integers, $\bar a_b$ will denote an integer such that $\bar a_ba\equiv 1\Mod b$.
Furthermore, when the context will be clear,  we will also use the simplified  notation $\bar a$ for this multiplicative inverse of $a$ to the implied modulus $b$, which is often the denominator of the fraction in whose numerator $\bar a$ appears.
The letter $p$ usually denotes a prime number. 
We adopt the convention throughout that $\e (t)=\exp (2i\pi t)$ unless $t$ contains an expression of the form $\frac{\bar a}b$ where $(a,b)>1$, in which case $\e(t)=0$. For example, the expression $\sum_{0\le n\le p-1} \e(\frac{\bar n\overline{(n-1)}}p)$ is to be interpreted as $0+0+\sum_{2\le n\le p-1} \e(\frac{\bar n_p\overline{(n-1)}_p}p)$.

\medskip


\section{Exponential sums involving multiplicative inverses}

We begin by establishing the following lemma on exponential sums (a slight variation of a result of Hooley~\cite[Lemma 3]{H67}), which is the crucial tool of our paper.

\begin{lemma}\label{IncSumBis}
For any $y<z$, $q\ge 2$, $t\in\Z$ and $(m,q)=1$ we have 
\begin{equation*}
\sum_{\substack{{y<n\le z}\\{(n,q)=1}\\{n\equiv u\Mod m}}}\e\Big (\frac{t\bar n_q}{q}\Big )
=\frac{z-y}{mq}\mu\Big (\frac{q}{(q,t)}\Big )\frac{\varphi (q)}{\varphi (q/(t,q))}+O\big( \sqrt{q(t,q)} \cdot \tau (q)\log q \big).
\end{equation*}
In particular,
\begin{equation} \label{IncSum eqn}
\sum_{\substack{{y<n\le z}\\{(n,q)=1}}}\e \Big ( \frac{t\bar n_q}{q}\Big )= \frac{z-y}{q} \mu\Big (\frac{q}{(q,t)}\Big )\frac{\varphi (q)}{\varphi (q/(t,q))}+O\big( \sqrt{q(t,q)}\cdot \tau (q)\log q \big).
\end{equation}
\end{lemma}

\noindent (While we have written the statement of the lemma, for ease of reference, with explicit subscripts on the multiplicative inverses~$\bar n_q$ and explicit coprimality conditions of summation, we immediately revert to our conventions of suppressing these notational signals.)

\begin{proof}
We begin by collecting the summands according to the value of $n\Mod q$, which we then detect using a further additive character:
\begin{align}
\sum_{\substack{{y<n\le z}\\{n\equiv u\Mod m}}}\e\Big (\frac{t\bar n}{q}\Big ) &= \sum_{a\Mod q} \e\Big(\frac{t\bar a}{q}\Big ) \sum_{\substack{{y<n\le z}\\n\equiv a\Mod q \\ {n\equiv u\Mod m}}} 1 \notag \\
&= \sum_{a\Mod q} \e\Big(\frac{t\bar a}{q}\Big ) \sum_{\substack{{y<n\le z} \\ {n\equiv u\Mod m}}} \frac1q \sum_{h=0}^{q-1} \e \Big( \frac{h(n-a)}q \Big) \notag \\
&= \frac1q \sum_{h=0}^{q-1} \bigg( \sum_{a\Mod q} \e\Big( \frac{t\bar a-ha}q \Big) \bigg) \bigg( \sum_{\substack{{y<n\le z} \\ {n\equiv u\Mod m}}} \e\Big( \frac{hn}q \Big) \bigg), \label{Kloostermania}
\end{align}
with the convention $\e (t\bar a/q)=0$ if $(a,q)>1$ as mentioned in the end of the introduction.
The $h=0$ summand contributes the main term:
\begin{align*}
\frac1q \bigg( \sum_{a\Mod q} \e\Big( \frac{t\bar a}q \Big) \bigg) \bigg( \sum_{\substack{{y<n\le z} \\ {n\equiv u\Mod m}}} 1 \bigg) &= \frac1q \bigg( \mu\Big (\frac{q}{(q,t)}\Big )\frac{\varphi (q)}{\varphi (q/(t,q))} \bigg) \bigg( \frac{z-y}m + O(1) \bigg) \\
&= \frac{z-y}{qm} \mu\Big (\frac{q}{(q,t)}\Big )\frac{\varphi (q)}{\varphi (q/(t,q))} + O(1),
\end{align*}
since the first sum on the left-hand side is a complete Ramanujan sum which has been evaluated classically (see~\cite[equation~(4.7)]{MV07}). As for the summands where $h\ne0$, the first inner sum on the right-hand side of equation~\eqref{Kloostermania} is a complete Kloosterman sum, which was shown by Hooley~\cite[Lemma 2]{H57} (using Weil's bounds for exponential sums) to be $\ll \sqrt{q(t,q)} \cdot \tau(q)$; therefore
\begin{align*}
\frac1q \sum_{h=1}^{q-1} \bigg( \sum_{a\Mod q} & \e\Big( \frac{t\bar a-ha}q \Big) \bigg) \bigg( \sum_{\substack{{y<n\le z} \\ {n\equiv u\Mod m}}} \e\Big( \frac{hn}q \Big) \bigg) \\
&\ll \sqrt{q(t,q)} \cdot \tau(q) \cdot \frac1q \sum_{h=1}^{q-1} \bigg| \sum_{\substack{{y<n\le z} \\ {n\equiv u\Mod m}}} \e\Big( \frac{hn}q \Big) \bigg| \\
&= \sqrt{q(t,q)} \cdot \tau(q) \cdot \frac1q \sum_{h=1}^{q-1} \bigg| \sum_{(y-u)/m<\lambda\le(z-u)/m} \e\Big( \frac{h(m\lambda+u)}q \Big) \bigg| \\
&= \sqrt{q(t,q)} \cdot \tau(q) \cdot \frac1q \sum_{\eta=1}^{q-1} \bigg| \e\Big( \frac{\eta\bar m u}q \Big) \sum_{(y-u)/m<\lambda\le(z-u)/m} \e\Big( \frac{\eta\lambda}q \Big) \bigg|,
\end{align*}
since $(m,q)=1$ and thus the change of variables $\eta=hm$ permutes the nonzero residue classes modulo~$q$. This last sum is a geometric series and is consequently $\ll \|\eta/q\|^{-1}$, where $\|u\|$ is the distance from $u$ to the nearest integer. Therefore we have the following bound for these summands:
\[
\ll \sqrt{q(t,q)} \cdot \tau(q) \cdot \frac2q \sum_{\eta=1}^{q/2} \frac q\eta \ll \sqrt{q(t,q)} \cdot \tau(q)\log q,
\]
as required.
\end{proof}

\section{Reducible quadratics (Theorem~\ref{1+1})}  \label{1+1 sec}

Throughout this section, we consider $f(n)=(an+b)(cn+d)$ with $a,b,c,d$ as in Theorem \ref{1+1}.
All implicit constants in this section may depend on~$f$, $h$, and (where appropriate)~$\varepsilon$.

Define $\Delta = ad-bc$.  
The simultaneous congruences $an+b\equiv cn+d\equiv 0\Mod p$ have a solution only when
$-b\bar a\equiv -d\bar c \Mod p$, or equivalently when $p\mid\Delta$.
We start with the case $(h,ac\Delta)=1$; in the following section we indicate how to deal with the general case.

\subsection{The case $(ac\Delta,h)=1$}

We begin by handling some coprimality problem between the denominators $n$ and $\Delta$.
\begin{lemma}\label{gDelta}
We may write
\begin{equation*}
S(f,x,h)=\sum_{\substack{{g\mid\Delta}\\{(g,ac)=1}}}\mu^2(g)\sum_{\substack {{m\le x/g}\\{(m,\Delta)=1}}}\sum_{
\substack{{r\Mod{mg}}\\{(ar+b)(cr+d)\equiv 0\Mod {mg}}}}\e\Big (\frac{hr}{mg}\Big ).
\end{equation*}
\end{lemma}

\begin{proof}
First we sort by $g=(n,\Delta)$:
\begin{equation*}
S(f,x,h)=\sum_{g\mid\Delta}\sum_{\substack {{m\le x/g}\\{(m,\Delta/g)=1}}}\sum_{
\substack{{r\Mod{mg}}\\{(ar+b)(cr+d)\equiv 0\Mod {mg}}}}\e\Big (\frac{hr}{mg}\Big ).
\end{equation*}
If $p\mid g\mid\Delta$ and $p\mid a$ then $p\mid c$ as well, in which case $(ax+b)(cx+d)\equiv bd\not\equiv 0\Mod p$ and there are no roots modulo~$p$. Similarly, if $p\mid(g,c)$, then $f$ has no root modulo~$p$. Thus we can add the condition of summation $(g,ac)=1$.

We can also assume that $p^2\nmid mg$ for every $p\mid\Delta$. Indeed, if $mg=qp^2$ for such a prime $p$, and if $r$ is a root of $(ax+b)(cx+d)$ modulo $qp^2$, then both $ar+b\equiv 0\Mod{p}$ and $cr+d\equiv 0\Mod p$ since $p\mid\Delta$; and then all of
$r+qp$, $r+2qp$, \dots, $r+(p-1)qp$ are roots of $(ax+b)(cx+d)$ modulo $qp^2$, since
\begin{equation*}\begin{split}
f(r+\lambda qp)&\equiv (ar+b)(cr+d)+\lambda qp (a(cr+d)+b(ar+b)) \equiv 0\Mod{qp^2}.
\end{split}
\end{equation*}
The exponential sum over these roots $r$, $r+qp$, \dots, $r+(p-1)qp$ vanishes (here we use the assumption $(h,\Delta)=1$, so that $p\nmid h$), which justifies omitting these terms.

In particular, we may assume that $g$ is squarefree (whence the introduction of the factor $\mu^2(g)$) and that $(m,g)=1$ (which combines with the existing condition $(m,\Delta/g)=1$ to yield simply $(m,\Delta)=1$), completing the proof.
\end{proof}

Next we separate the congruence conditions between the two factors of~$f$.

\begin{lemma}\label{rg} Suppose that $g\mid\Delta$ is squarefree, $(g,c)=1$, and $(m,\Delta)=1$.
The roots of $(ar+b)(cr+d)\equiv 0\Mod {mg}$ are in one-to-one correspondence with the factorizations $k\ell=m$
with $(k,\ell )=1$, $(k,a)=1$, and $(\ell ,c)=1$. The root $r$ corresponds to the solution of the system of congruences
\begin{equation*}
\begin{split}
ar+b & \equiv 0 \Mod k\\
cr+d & \equiv 0\Mod\ell\\
r&\equiv r_g\Mod g,\\
\end{split}
\end{equation*}
where $r_g$ is the residue class $r_g\equiv -d\bar c\Mod g$.
\end{lemma}

\begin{proof}
It is straightforward to check that the factorization corresponding to a root $r$ of $(ar+b)(cr+d)\equiv 0\Mod {mg}$ is $k = (ar+b,m)$ and $\ell=(cr+d,m)$, and that this correspondence is the inverse function to the correspondence described in the statement of the lemma.
\end{proof}

By the Chinese remainder theorem, the solution of the system of congruences given in Lemma~\ref{rg} can be written as 
\begin{equation*}r=-b\bar a_k\ell\bar\ell _kg\bar g_k-d\bar c_\ell k\bar k_\ell g\bar g_\ell +r_gk\bar k_g\ell\bar\ell_g.
\end{equation*}
We thus obtain 
\begin{equation*}
\begin{split}
S(f,x,h)&=\!\!\!\sum_{\substack{{g\mid\Delta}\\{(g,ac)=1}}}\!\!\!\mu ^2(g)\sum_{\substack {{m\le x/g}\\{(m,\Delta)=1}}}\!\!\!\!\!\!\!
\sum_{\substack{{k\ell =m}\\{(k,\ell )=1}\\{(k,a)=(\ell ,c)=1}}}\!\!\!\!\!\!\!\e\Big (\frac{
-hb\bar a_k\ell\bar\ell _kg\bar g_k-hd\bar c_\ell k\bar k_\ell g\bar g_\ell +hr_gk\bar k_g\ell\bar\ell_g}{mg}\Big )\\
& =\sum_{\substack{{g\mid\Delta}\\{(g,ac)=1}}}\mu ^2(g)\sum_{\substack{{k\ell \le x/g}\\{(k,\ell )=1}\\{(k,a)=(\ell ,c)=1}\\{(k\ell,\Delta)=1}}}\e\Big (\frac{
-hb\bar a_k\bar\ell _k\bar g_k}{k}\Big )\e\Big (\frac{-hd\bar c_\ell \bar k_\ell \bar g_\ell}{\ell}\Big )
\e\Big (\frac{hr_g\bar k_g\bar\ell_g}{g}\Big ).
\end{split}
\end{equation*}
We split this into two sums $S(f,x)=S_1(x)+S_2(x)$, where $k\le\sqrt{x/g}$ in $S_1(x)$ and $k>\sqrt{x/g}$ in 
$S_2 (x)$.

In $S_1(x)$, we use the standard ``inversion formula'' (obtained with B\'ezout's identity) for $(u,v)=1$,
\begin{equation} \label{inversion formula}
\frac{\bar u}{v}+\frac{\bar v}{u}\equiv \frac{1}{uv}\Mod 1.
\end{equation}
This formula, used on the second exponential factor in the inner sum above, allows us to move $\ell$ out of the denominators of the exponential terms and write $S_1(x)$~as
\begin{equation*}
\sum_{\substack{{g\mid\Delta}\\{(g,ac)=1}}}\mu ^2(g)
\sum_{\substack{{k\le\sqrt{x/g}}\\{(k,a\Delta)=1}}}\sum_{\substack{{\ell\le x/gk}\\{(\ell, k)=(\ell ,c\Delta)=1}}}
\e\Big (\frac{-hb\bar a_k\bar\ell _k\bar g_k}{k}\Big )\e\Big (\frac{hd\bar\ell _{cgk}}{cgk}\Big )\e\Big (\frac{-hd}{cgk\ell}\Big )\e\Big (\frac{hr_g\bar k_g\bar\ell _g}{g}\Big ).
\end{equation*}
Next we remark that $\e (-d/cgk\ell)=1+O(1/gk\ell)$. This effect of this error term is sufficiently small:
\begin{equation*}\begin{split}
S_1(x)&=\sum_{\substack{{g\mid\Delta}\\{(g,ac)=1}}}\mu ^2(g)
\sum_{\substack{{k\le\sqrt{x/g}}\\{(k,a\Delta)=1}}}\sum_{\substack{{\ell\le x/gk}\\{(\ell, k)=(\ell ,c\Delta)=1}}}
\e\Big (\frac{(-bcg\bar a_k\bar g_k+d+ckr_g\bar k_g)h\bar\ell_{cgk}}{cgk}\Big )\\
&\qquad{}+O((\log x)^2),\\
\end{split}
\end{equation*}
since $\bar\ell_k\equiv\bar\ell _{cgk}\Mod k$ and $\bar\ell _g\equiv\bar\ell_{cgk}\Mod g$. (We remind the reader that the implicit constants in this section may depend on $f$, $h$, and~$\varepsilon$.)

We rewrite the inner sum over $\ell$ as
\begin{equation}\label{sigmal}
\sum_{\substack{{\ell\le x/gk}\\{(\ell ,ck\Delta)=1}}}\e\Big (\frac{\frac{\Delta}{g}(-bcg\bar a_k\bar g_k+d+ckr_g\bar k_g)h\bar \ell_{ck\Delta}}{ck\Delta}\Big ),
\end{equation}
and then apply equation~\eqref{IncSum eqn} with $q=ck\Delta$ and $t=ht'$, where $t'=\frac{\Delta}{g}(-bcg\bar a_k\bar g_k+d+ckr_g\bar k_g)$. Since $d+ckr_g\bar k_g\equiv d+cr_g\equiv 0\Mod g$, we see that $t'$ is a multiple of $\Delta$ but is relatively prime to both $c$ and~$k$.
We thus obtain
\begin{equation*}\begin{split}
S_1(x)&=x\sum_{\substack{{g\mid\Delta}\\{(g,ac)=1}}}\mu ^2(g)\sum_{\substack{{k\le\sqrt{x/g}}\\{(k,ac\Delta)=1}}}
\frac{\mu \big (\frac{ck}{(ck,h)}\big )}{cg\Delta k^2}\frac{\varphi (ck\Delta)}{\varphi (ck/(ck,h))}\\
&\qquad{}+O\Big (\sum_{\substack{{g\mid\Delta}\\
{(g,ac)=1}}}\mu^2(g)\sum_{\substack{{k\le\sqrt{x/g}}\\{(k,a\Delta)=1}}}(k(k,h))^{1/2+\varepsilon}\Big ).
\end{split}
\end{equation*}
In the previous formula we could insert the condition $(c,k)=1$ in the main term, because the fact that $(h,c)=1$ means that $\mu (ck/(ck,h))=0$
if $(c,k)>1$.
The error term is $O(x^{3/4+\varepsilon})$.
Let $T_1(x)$ denote  the main term. By the coprimality conditions among $a,c,\Delta,k$,  we get
\begin{equation*}
T_1(x)=x \sum_{\substack{{g\mid\Delta}\\{(g,ac)=1}}}\frac{\mu ^2(g)\mu \big ( \frac{c}{(c,h)}\big )}{cg\Delta}
\frac{\varphi (c\Delta)}{\varphi \big (\frac{c}{(c,h)}\big )}\sum_{\substack{{k\le\sqrt{x/g}}\\{(k,ac\Delta)=1}}}
\frac{\mu \big (\frac{k}{(k,h)}\big )}{k^2}\frac{\varphi (k)}{\varphi \big (\frac{k}{(k,h)}\big )}.
\end{equation*}
Let $\lambda_h$ denote the function in the inner sum over $k$:
\begin{equation*}
\lambda_h(k)=\frac{\mu \big (\frac{k}{(k,h)}\big )}{k^2}\frac{\varphi (k)}{\varphi \big (\frac{k}{(k,h)}\big )}.
\end{equation*}
Note that $\lambda_h(k)$ is a multiplicative function of~$k$.
When $p\nmid h$,  we have 
$\lambda _h (p)=-1/p^2$ and $\lambda_h (p^\nu)=0$ for $\nu\ge 2$.
When $p\mid h$, if $v_p(h)$ is the $p$-adic valuation of~$h$, we have:
\begin{equation*}
\lambda_h(p^\nu)=\begin{cases}\frac{\varphi (p^\nu)}{p^{2\nu}} & \text{if }\nu\le v_p (h)\\
-\frac{1}{p^{\nu +1}} & \text{if } \nu = v_p (h)+1\\
0 & \text{if }\nu\ge v_p (h)+2;
\end{cases}
\end{equation*}
in particular, $|\lambda_h(p^\nu)| \le p^{v_p(h)}/p^{2\nu}$ always, which implies that $|\lambda_h(n)| \le h/n^2 \ll 1/n^2$ since $h$ is fixed. We deduce that 
\begin{equation*}
\sum_{\substack{{k\le\sqrt{x/g}}\\{(k,ac\Delta)=1}}}\lambda_h(k)=\sum_{(k,ac\Delta)=1}
\lambda_h(k)+O(hx^{-1/2+\varepsilon}),
\end{equation*}
from which we obtain
$T_1(x)=C_1(h,a,c,\Delta)x$
with 
\begin{equation}\label{C1}
\begin{split}
C_1 (h,a,c,\Delta)&=\frac{6}{\pi ^2}\frac{\mu\big (\frac{c}{(c,h)}\big )\varphi (c\Delta)}{c\Delta \varphi \big (\frac{c}{(c,h)}\big )}\prod_{\substack{{ p\mid \Delta}\\ {(p, ac)=1}}}\left ( 1+\frac{1}{p}\right )\\
&\prod_{p\mid ahc\Delta}\Big ( 1-\frac{1}{p^2}\Big )^{-1}\prod_{\substack{{p\mid h}\\{(p,ac\Delta)=1}}}
\Big ( 1+\frac{1}{p}-\frac{2}{p^{v_p(h)+1}}\Big ).\\
\end{split}
\end{equation}
In the particular case $h=1$, this gives 
\begin{equation*}
S_1(x)=\frac{6x}{\pi ^2}\frac{\mu (c)}{c}\prod_{p\mid ac}\Big ( 1-\frac{1}{p^2}\Big )^{-1}+O(x^{3/4+\varepsilon}).
\end{equation*}
In this last computation, we have used the fact that if $p\mid (c,\Delta)$, then $p\mid a$.
In the same way, we derive the asymptotic formula
\begin{equation}\label{C2}
S_2(x)=C_1(h,c,a,\Delta)x+O(x^{3/4+\varepsilon})
\end{equation}
which, when $h=1$, simplies to
\begin{equation*}
S_2(x)=\frac{6x}{\pi ^2}\frac{\mu (a)}{a}\prod_{p\mid ac}\Big ( 1-\frac{1}{p^2}\Big )^{-1}+O(x^{3/4+\varepsilon}).
\end{equation*}
This finishes the proof of Theorem \ref{1+1} in the case $(h,ac\Delta)=1$.

\subsection{The case $(h,ac\Delta)\ne1$}\label{1+1General}

The main difference when dealing with the case $(h,ac\Delta)\ne1$ is that we lose some cancellation observed in the proof 
of Lemma~\ref{gDelta}. We use the notation $m\mid n^\infty$ to indicate that $p\mid m\Rightarrow p\mid n$, and we sort the summands according to the prime factors they share with $(h,ac\Delta)$:
\begin{equation}\label{hacD}
S(f,x,h)=\sum_{\delta\mid (h,ac\Delta)^\infty}\sum_{\substack{{n\le x/\delta}\\ {(n,(h,ac\Delta))=1}}}
\sum_{\substack{{0\le r<\delta n}\\ {f(r)\equiv 0\Mod{\delta n}}}}\e\Big (\frac{hr}{\delta n}\Big ).
\end{equation}
Since $\deg f=2$, the number of roots of $f(r)\equiv 0\Mod{\delta n}$ is less than $2^{\omega (\delta n)}$.
Defining $B=x^{1/5-\varepsilon}$ for some sufficiently small $\varepsilon >0$, we split this sum as $S(f,x,h)=S_>(x,h)+S_\le(x,h)$, where $\delta >B$ in $S_>(x,h)$ and $\delta \le B$ in $S_\le(x,h)$.
When $\delta$ is large, a trivial upper bound is sufficient:
\begin{equation}\label{>}
\begin{split}
S_>(x,h)&\ll x\log x\sum_{\substack{{\delta \mid (h,ac\Delta)^\infty}\\{\delta >B}}}\frac{2^{\omega (\delta)}}{\delta}\\
&\ll x\log x\sum_{\substack{{\delta \mid (h,ac\Delta)^\infty}}}\frac{\delta^{\varepsilon_1}}{\delta} \bigg( \frac\delta B\bigg)^{1-2\varepsilon_1} = \frac{x\log x}{B^{1-2\varepsilon_1}}\prod_{p\mid (h,ac\Delta)}\Big ( 1-\frac{1}{p^{\varepsilon_1}}\Big )^{-1}.\\
\end{split}
\end{equation}

When $\delta$ is not large, we adapt the method of the previous section.
Since $(n,\delta)=1$, we have by the Chinese remainder theorem:
\begin{equation*}\label{<}
S_\le(x,h)=\sum_{\substack{{\delta\mid (h,ac\Delta)^\infty}\\{\delta\le B}}}\sum_{\substack{{n\le x/\delta}\\ {(n,(h,ac\Delta))=1}}} \bigg( \sum_{\substack{{0\le r_0<\delta}\\{ f(r_0n)\equiv 0\Mod\delta}}}e\Big (\frac{hr_0}{\delta}\Big) \bigg) \bigg( \sum_{\substack{{0\le r_1<n}\\{ f(r_1\delta)\equiv 0\Mod n}}}e\Big (\frac{hr_1}{n}\Big ) \bigg).
\end{equation*}
In order to suppress the dependence of $n$ in the summation in $r_0$,
we split this sum according to $n$ modulo~$\delta$:
\begin{equation*}\begin{split}
S_\le(x,h)&=\sum_{\substack{{\delta\mid (h,ac\Delta)^\infty}\\{\delta\le B}}}\sum_{\substack{{0\le \alpha <\delta}\\{(\alpha ,\delta )=1}}} \bigg(
\sum_{\substack{{0\le r_0<\delta}\\{ f(r_0\alpha)\equiv 0\Mod\delta}}}e\Big (\frac{hr_0}{\delta}\Big) \bigg)\\
&\qquad{}\times \bigg(
\sum_{\substack{{n\le x/\delta}\\ {(n,(h,ac\Delta))=1}\\ {n\equiv\alpha\Mod\delta}}}
\sum_{\substack{{0\le r_1<n}\\{ f(r_1\delta)\equiv 0\Mod n}}}e\Big (\frac{hr_1}{n}\Big) \bigg).\\
\end{split}
\end{equation*}
Let $G(\alpha,\delta)$ denote the above sum over $r_0$ and $H(\alpha,\delta)$ the double sum over $n$ and $r_1$, so that 
\begin{equation}\label{GHk}
S_\le(x,h)=\sum_{\substack{{\delta\mid (h,ac\Delta)^\infty}\\{\delta\le B}}}\sum_{\substack{{0\le \alpha <\delta}\\{(\alpha ,\delta )=1}}}G(\alpha ,\delta)H(\alpha,\delta).
\end{equation}
It is convenient to notice now that the number of summands in the sum defining $G(\alpha ,\delta)$ is at most the number of roots of $f(r)\Mod\delta$, and hence that sum is bounded by a constant depending on $f$: indeed, a bound for that number of roots by Sitar and the second author~\cite[Lemma 3.4]{MS11} implies that
\begin{equation*}
|G(\alpha ,\delta)|\le \Delta^{1/2}2^{\omega(\delta)} \le \Delta^{1/2}2^{\omega(ac\Delta)}.
\end{equation*}

We handle the term $H(\alpha ,\delta)$ in the same way as in the previous section. We write $\Delta=\Delta_1\Delta_2$ with $(\Delta_1,h)=1$ and $\Delta_2\mid (h,\Delta)^\infty$.
A version of Lemma~\ref{gDelta} with $\Delta_1$ in place of $\Delta$ gives
\begin{equation*}
H(\alpha ,\delta)=\sum_{\substack{{g\mid\Delta_1}\\{(g,ac)=1}\\}}\mu^2(g)\sum_{\substack{{m\le x/g\delta}\\ {(m,\Delta)=1}\\ {gm\equiv\alpha\Mod \delta}}}\sum_{\substack{{0\le r_1<mg}\\ {f(r_1\delta)\equiv 0\Mod {mg}}}}\e\Big (\frac{hr_1}{mg}\Big ).
\end{equation*}

The corresponding sum $S_1(x)$ is then of type 
\begin{equation*}\begin{split}
S_1(x)&=\sum_{\substack{{g\mid\Delta_1}\\{(g,ac)=1}}}\mu ^2(g)
\sum_{\substack{{k\le\sqrt{x/\delta g}}\\{(k,a\Delta)=1}}}\sum_{\substack{{\ell\le x/gk\delta}\\{(\ell, k)=(\ell ,c\Delta)=1}\\{
gk\ell\equiv\alpha\Mod\delta}}}
\e\Big (\frac{(-bcg\bar a_k\bar g_k+d+ckr_g\bar k_g)h\bar\ell_{cgk}}{cgk}\Big )\\
&\qquad{}+O((\log x)^2).
\end{split}
\end{equation*}
Since $(\alpha ,\delta)=1$, the condition $gk\ell\equiv \alpha\Mod\delta$ implies that $(g,\delta)=1$.
The analogue of equation~\eqref{sigmal} is now
\begin{equation*}
\Sigma_1:=\sum_{\substack{{\ell\le x/gk}\\{(\ell ,ck\Delta)=1}\\{gk\ell\equiv\alpha\Mod\delta}}}\e\Big (\frac{\frac{\Delta_1}{g}(-bcg\bar a_k\bar g_k+d+ckr_g\bar k_g)h\bar \ell_{ck\Delta}}{ck\Delta_1}\Big ),
\end{equation*}
in which the only real difference from before is the congruence $gk\ell\equiv\alpha\Mod\delta$. 
 Let $\delta_2$ be the least common multiple $\delta_2=[\delta ,\Delta_2]$; we still have
$(\delta_2 , ck\Delta_1)=1$. The two conditions $g\ell k\equiv\alpha\Mod\delta$ and 
$(\ell ,\Delta_2)=1$ can be expressed using congruences modulo~$\delta _2$:
\begin{equation*}
\Sigma_1=\sum_{\substack{{0\le \beta <\delta_2/\delta}\\{(\alpha \bar k_\delta +\delta \beta,\delta_2)=1}}}
\sum_{\substack{{\ell\le x/gk}\\{(\ell ,ck\Delta_1)=1}\\{\ell\equiv\alpha\bar k_\delta+\beta\delta\Mod{\delta_2}}}}\e\Big (\frac{\frac{\Delta_1}{g}(-bcg\bar a_k\bar g_k+d+ckr_g\bar k_g)h\bar \ell_{ck\Delta}}{ck\Delta_1}\Big ).
\end{equation*}

We apply Lemma~\ref{IncSumBis} in the same way as in the case $(\Delta, h)=1$ to obtain:
\begin{equation*}
S_\le(x,h)=xC(h,a,c,\Delta_1)\sum_{\substack{{\delta\mid (h,ac\Delta)^\infty}\\{\delta\le B}}}\sum_{\substack{{0\le \alpha <\delta}\\{(\alpha ,\delta )=1}}} \frac{G(\alpha ,\delta)}{\delta ^2}+O\Big (x^{3/4 +\varepsilon}\sum_{\substack{{\delta\mid (h,ac\Delta)^\infty}\\{\delta\le B}}}\delta^{1/4}\Big ).
\end{equation*}
If we open the sum $G(\alpha ,\delta)$ and exchange the order of summation with $\alpha$, we find Ramanujan sums:
\begin{equation*}\begin{split}
\sum_{\substack{{0\le \alpha <\delta}\\{(\alpha ,\delta )=1}}} G(\alpha ,\delta)&=
\sum_{\substack{{0\le r_0<\delta}\\ {f(r_0)\equiv 0\Mod \delta}}}\sum_{\substack{{0\le \alpha <\delta}\\{(\alpha ,\delta )=1}}}
\e\Big (\frac{hr_0\bar\alpha}{\delta}\Big )\\
& = \sum_{\substack{{0\le r_0<\delta}\\ {f(r_0)\equiv 0\Mod \delta}}}\mu\Big ( \frac{\delta}
{(hr_0,\delta)}\Big )\frac{\varphi (\delta)}{\varphi\Big (\frac{\delta}{(hr_0,\delta)}\Big )}.
\end{split}
\end{equation*}

Since $G(\alpha ,\delta ) \ll 1$ as previously remarked, the sum over $\delta$ in the main term converges as $B$ tends to $\infty$ (note that the sum is not over all integers $\delta$ but rather only those integers with prime factors in a fixed finite set, which is a very sparse sequence), and the error resulting from replacing the finite sum over~$\delta$ with the infinite series is $\ll x/B^{1-\varepsilon}$. In the error term,
the $\delta ^{1/4}$ roughly comes from the summation of the $k^{1/2}$ with $k\le (x/g\delta )^{1/2}$
and with a trivial summation of the sum over~$\alpha$.
This error term is $O(x^{3/4+\varepsilon}B^{1/4})$. 

This ends the proof of Theorem \ref{1+1}, with
\begin{equation} \label{C1 general}
C(f,h) = C(h,a,c,\Delta_1)\sum_{\substack{{\delta\mid (h,ac\Delta)^\infty}}} \frac1{\delta ^2} \sum_{\substack{{0\le r_0<\delta}\\ {f(r_0)\equiv 0\Mod \delta}}}\mu\Big ( \frac{\delta}
{(hr_0,\delta)}\Big )\frac{\varphi (\delta)}{\varphi\Big (\frac{\delta}{(hr_0,\delta)}\Big )},
\end{equation} 
where $C(h,a,c,\Delta_1)$ is as in equation~\eqref{C1}.

\section{Linear times an irreducible quadratic (Theorem~\ref{1+2})} \label{1+2 sec}

In this section we prove Theorem \ref{1+2} concerning the polynomial $f(t)=t(t^2+1)$.

\subsection{First step: splitting $S(f,x)$}

Since $(n,n^2+1)=1$ we can write $S(f,x)$ as a sort of convolution, as in the previous section. The following lemma is elementary:

\begin{lemma}\label{kl}
Let $g(t)$ be  a polynomial with integer coefficients with $g(0)=\pm 1$, and let $m$ be a positive integer. The roots of $rg(r)\equiv 0\Mod m$ are in one-to-one correspondence with the factorizations $k\ell =m$
with $(k,\ell )=1$ and corresponding roots $v$ of $g(v)\equiv 0\Mod m$. The root $r$ corresponds to the solution modulo $m$ of the system of congruences
\begin{equation*}
\begin{split}
r& \equiv 0 \Mod k\\
r&\equiv v \Mod \ell.\\
\end{split}
\end{equation*}
\end{lemma}

We have by Lemma \ref{kl}
\begin{equation*}
S(f,x)=\sum_{n\le x}\sum_{\substack{{k\ell =n}\\{(k,\ell )=1}}}\sum_{\substack{{0\le v <\ell}\\{v^2+1\equiv 0\Mod\ell}}}\e\Big (\frac{\bar kv}{\ell}\Big ).
\end{equation*}
Let $y_1=x^{1/3}(\log x)^{-A}$ and $y_2=x^{1/3}(\log x)^B$ with $A,B>0$ to be chosen.
We split the sum $S(x)$ into three parts according to the size of~$k$: $S(f,x)=S_1(x)+S_2(x)+S_3(x)$ with
\begin{align*}
S_1(x)&=\sum_{y_2<k\le x}\sum_{\substack{{\ell \le x/k}\\{(k,\ell )=1}}}\sum_{\substack{{0\le v <\ell}\\{v^2+1\equiv 0\Mod\ell}}}\e\Big (\frac{\bar kv}{\ell}\Big ) \\
S_2(x)&=\sum_{y_1<k\le y_2}\sum_{\substack{{\ell \le x/k}\\{(k,\ell )=1}}}\sum_{\substack{{0\le v <\ell}\\{v^2+1\equiv 0\Mod\ell}}}\e\Big (\frac{\bar kv}{\ell}\Big ) \\
S_3(x)&=\sum_{1\le k\le y_1}\sum_{\substack{{\ell \le x/k}\\{(k,\ell )=1}}}\sum_{\substack{{0\le v <\ell}\\{v^2+1\equiv 0\Mod\ell}}}\e\Big (\frac{\bar kv}{\ell}\Big ).
\end{align*}

We shall see momentarily that in $S_1(x)$, it is possible to use equation~\eqref{IncSum eqn} in the same way as in the proof of Theorem \ref{1+1}; the main term in our asymptotic formula for $S(f,x)$ arises from this sum.
This approach works only when $k$ is sufficiently large (or $\ell$ sufficiently small), which is to say when $k$ is slightly bigger than $x^{1/3}$. 
This is our motivation for the choice of~$y_2$.

The converse is true for $S_3(x)$, in which $\ell$ is the largest parameter. In this case we use the fact that the second factor is quadratic. We can 
apply a lemma of Gauss on the correspondence of the roots of $n^2+1\equiv 0\Mod \ell$ and certain representations $\ell =r^2+s^2$ as the sum of two squares. This approach works when $k=o(x^{1/3})$, which is why we choose $y_1$ close to $x^{1/3}$.

The remaining range $k\in {}]y_1,y_2]$ is covered by a direct application of Hooley's result~\cite{H64}. Since $y_1$ and $y_2$ are close together, Hooley's general bound applied to the irreducible polynomial $X^2+1$  is sufficient.

\subsection{The first two sums}

In the sum $S_1(x)$, the variable $k$ is large and thus we arrange for some cancellation in the sum over this variable:
\begin{equation*}
S_1 (x)=\sum_{\ell <x/y_2}\sum_{\substack{{0\le v <\ell}\\{v^2+1\equiv 0\Mod\ell}}}\sum_{\substack{{y_2<k\le x/\ell}\\{(k,\ell )=1}}}\e\Big (\frac{\bar kv}{\ell}\Big ).
\end{equation*}
%
%
Let $\varrho (m)$ denote the number of roots modulo $m$ of the polynomial $n^2+1$:
\begin{equation}\label{rho}
\varrho (m)=\#\{ 0\le  v<m \colon v^2+1\equiv 0\Mod m\}.
\end{equation}
For any $\ell <x/y_2$ and any $0\le v<\ell$ with $v^2+1\equiv 0\Mod \ell$, we apply equation~\eqref{IncSum eqn} 
to bound the inner sum in~$k$:
\begin{equation*}
\begin{split}
S_1 (x)& =\sum_{\ell\le x/y_2}\sum_{\substack{{0\le v <\ell}\\{v^2+1\equiv 0\Mod\ell}}}\Big (\Big (\frac{x/\ell-y_2}{\ell}\Big )\mu (\ell)+O(\sqrt{\ell }\cdot \tau (\ell)\log \ell)\Big )\\
& =x\sum_{\ell\le x/y_2}\frac{\mu (\ell)\varrho (\ell)}{\ell ^2}-y_2\sum_{\ell\le x/y_2}\frac{\mu (\ell)}{\ell}+O\Big (\sum_{\ell\le x/y_2}\sqrt{\ell}\cdot \tau (\ell)\varrho (\ell)\log \ell\Big )\\
&=x\prod_p\Big (1-\frac{\varrho (p)}{p^2}\Big )+O(y_2x^\varepsilon +(x/y_2)^{3/2}(\log x )^5).\\
&=x\prod_p\Big (1-\frac{\varrho (p)}{p^2}\Big )+O(x(\log x)^{5-B}).
\end{split}
\end{equation*}
In the third equality above, we used the calculation
\begin{equation*}
\begin{split}
\sum_{\ell\le x/y_2}\sqrt{\ell} \cdot \tau (\ell)\varrho (\ell)\log \ell&\ll \Big (\frac{x}{y_2}\Big )^{3/2}\log x\sum_{\ell\le x/y_2}\frac{\tau (\ell)\varrho (\ell)}{\ell}\\
&\ll  \Big (\frac{x}{y_2}\Big )^{3/2}\log x\prod_{p\le x/y_2}\Big (1+\sum_{k=1}^\infty \frac{\tau (p^k)\varrho (p^k)}{p^k}\Big ),
\end{split}
\end{equation*}
followed by the fact that $\varrho (p^k)\le 2$ for all prime powers~$p^k$. (We could in fact replace the exponent $5-B$ by $2-3B/2$, using the fact that $\varrho (p^k)=0$ if $p\equiv 3\Mod 4$, but that improvement is not significant for our purposes.)

In the sum $S_2 (x)$, the variable $k$ is in a crucial range (corresponding to when the size of $k$ is close to $\sqrt{\ell}$) where the methods for both $S_1(x)$ and $S_3(x)$ fail.
The bound for $S_2(x)$ will be a direct consequence of the work of Hooley:
\begin{lemma}
Let $P(X)\in\Z [X]$ be an irreducible polynomial of degree $n\ge 2$.
If $hk\not =0$ then we have
\begin{equation*}
\sum_{\substack{{\ell\le x }\\{(\ell ,k)=1}}}\sum_{\substack{{v\Mod \ell}\\{P(v)\equiv 0\Mod \ell}}}\e\Big (\frac{h\bar k v}{\ell}\Big )\ll_{h,P} x\frac{(\log\log x)^{(n^2+1)/2}}{(\log x)^{\delta _n}},
\end{equation*}
where 
$\delta _n=(n-\sqrt n)/{n!}$.
\end{lemma}

\noindent The case $k=1$ is~\cite[Theorem 1]{H64}; the proof can be adapted with no difficulty for all $k\in\N$ and provides then a result that is uniform in~$k$.
The dependence on~$h$ in this result of Hooley is due only to the appearance of $(h, \ell)$ in certain intermediate computations.

We apply this lemma with $P(X)=X^2+1$ and replacing $x$ by $x/k$:
\begin{equation*}
\begin{split}
S_2 (x)&\ll\sum_{y_1\le k\le y_2}\frac{x}{k}\frac{(\log\log x)^{5/2}}{(\log x)^{1-1/\sqrt2}} \ll x\frac{(\log\log x)^{7/2}}{(\log x)^{1-1/\sqrt2}},
\end{split}
\end{equation*}
using the fact that $\log (y_2/y_1) \ll \log\log x$.
 
\subsection{The sum $S_3(x)$}

In this section we use the special shape of the polynomial $n^2+1$ to find an upper bound for $S_3(x)$.
Following the ideas of the two articles of Hooley \cite{H63,H67} concerning $\tau (n^2+1)$ and $P^+ (n^2+1)$ (the number of divisors of $n^2+1$ and the largest prime factor of $n^2+1$, respectively), we employ the Gauss correspondence between the roots of $v^2+1\equiv 0\Mod \ell$ and certain representations
of $\ell =r^2+s^2$ as the sum of two squares. Indeed for such integers $r,s$ with $(rs,\ell)=1$,
we have $(\bar r s)^2+1\equiv 0\Mod\ell$.
The parameter $\bar k$ in the exponential gives rise to some coprimality problems.
The first author~\cite{Da96} resolved such a difficulty when $k$ is squarefree; in equation~\eqref{k1k2} we will use an elegant formula of Wu and Xi~\cite{WX}
to handle the general case.

As in the proof of Theorem \ref{1+1}, in the following argument the condition  $k_2=(k,r^\infty)$ means that $p\mid k_2\Rightarrow p\mid r$ and $(k/k_2, r)=1$.

\begin{lemma}
\label{vrs}
For $\ell >1$, there is a one-to-one correspondence between the representations of $\ell$  by the form $\ell =r^2+s^2$ with $(r,s)=1$, $r>0,s>0$
and the solutions of the congruence $v^2+1\equiv 0\Mod \ell$.   This bijection is given by:
\begin{equation*}
\frac{v}{\ell}=\frac{\bar s}{r}-\frac{s}{r(r^2+s^2)} \Mod 1.
\end{equation*}
For $k\ge 1$, $k=k_1k_2$ with $k_2=(k,r^\infty)$ we have
\begin{equation}\label{k1k2}
\frac{\bar kv}{\ell }=\frac{-r\overline{k_2(r^2+s^2)}}{k_1s}+\frac{r}{ks(r^2+s^2)}-\frac{r\overline{k_1s(r^2+s^2)}}{k_2}\Mod 1.
\end{equation}
\end{lemma} 

\noindent The first part of this lemma is proved in detail in the book of Smith~\cite[Art.\ 86]{Sm65}, while equation~\eqref{k1k2} is \cite[Lemma~7.4]{WX}.

By this lemma, still using the notation $k=k_1k_2$, we have
\begin{equation*}
S_3(x)=\sum_{\substack{{k_1k_2\le y_1}\\{(k_1,k_2)=1}}}\sum_{\substack{{r^2+s^2\le x/(k_1k_2)}\\{(r,s)=1,\, r>0,\, s>0}\\ {(k_1k_2,r^\infty )=k_2}}}
\e\Big ( \frac{-r\overline{k_2(r^2+s^2)}}{k_1s}+\frac{r}{ks(r^2+s^2)}-\frac{r\overline{k_1s(r^2+s^2)}}{k_2}\Big ).
\end{equation*}
First we remove the term $\e \big ( {r}/{ks(r^2+s^2)}\big )$: since
\begin{equation*}
\e \Big ( \frac{r}{ks(r^2+s^2)}\Big )=1+O\Big (\frac{1}{ksr}\Big ),
\end{equation*}
replacing this term by $1$ results in a corresponding error in $S_3 (x)$ that is $O((\log x)^4)$.

Following the notation of several authors,  we denote by $(k_1s)^\flat$ and $(k_1s)^\sharp$ the squarefree and squarefull part, respectively, of $k_1s$.
Since $((k_1s)^\flat , (k_1s)^\sharp )=1$, we can use  the Chinese remainder theorem  as in the proof of Theorem \ref{1+1}:
\begin{equation*}
\frac{1}{k_1s}\equiv\frac{\overline{(k_1s)^\sharp}}{(k_1s)^\flat}+\frac{\overline{(k_1s)^\flat}}{(k_1s)^\sharp} \Mod 1.
\end{equation*}
Inserting this in $S_3(x)$, we obtain
\begin{equation}\label{prodS3}
S_3(x)=\sum_{\substack{{k_1k_2\le y_1}\\{(k_1,k_2)=1}}}\sum_{\substack{{r^2+s^2\le x/(k_1k_2)}\\{(r,s)=1, r>0,s>0}\\ {(k_1k_2,r^\infty )=k_2}}}
K(r)W(r) +O((\log x)^4),
\end{equation}
with
\begin{align*}
K(r)&=\e\Big ( \frac{-r\overline{k_2(r^2+s^2)(k_1s)^\sharp}}{(k_1s)^\flat}\Big ) \\
W(r)&=\e \Big ( \frac{-r\overline{k_2(r^2+s^2)(k_1s)^\flat}}{(k_1s)^\sharp}-\frac{r\overline{k_1s(r^2+s^2)}}{k_2}\Big ).
\end{align*}
Let $S_4(x)$ denote the contribution to $S_3(x)$ of the $k_1,k_2,r,s$ such that $ (k_1s)^\sharp>(\log x)^{45}$ or $k_2>(\log x)^5$,
and $S_5(x)$ the remaining contribution, that is, the contribution of the $k_1,k_2,r,s$ such that
$k_2 \le (\log x)^5$ and $(k_1s)^\sharp\le  (\log x)^{45}$.

First we  prove that 
\begin{equation}\label{S4}
S_4(x)\ll x(\log x)^{-3}.
\end{equation}
We remark that if   $m^2$ is the largest square divisor of $(k_1s)^\sharp$ then $m^2\ge ((k_1s)^\sharp)^{2/3}$.  
We deduce that when $(k_1s)^\sharp >(\log x)^{45}$,  there exists $m>(\log x)^{15}$ such that $m^2\mid (k_1s)^\sharp$.
We can write this divisor in the following way: $m^2=u^2v^2w^2$ with $u^2\mid k_1$, $v^2\mid s$, and $w\mid (k_1,s)$. Thus we have
$\max \{u^2,v^2, w^2\}\ge m^{2/3}$ and  there exists $d>(\log x)^5$ such that $d^2\mid k_1$ or $d^2\mid s$ or $d\mid (k_1,s)$.
In the first case (when $d^2\mid k_1$), the contribution of the $k_1,k_2,r,s$ is less than
\begin{equation*}
\sum_{(\log x)^5<d\le y_1}
\sum_{d^2k_1k_2<y_1}
\sum_{\max \{r,s\}\ll x^{1/2}/(d^2k_1k_2)^{1/2}}1\ll x(\log x)^{-3}.
\end{equation*}
Similarly, in the second case (when $d^2\mid s$), we have a contribution less than
\begin{equation*}
\sum_{k_1k_2<y_1}\sum_{r\ll (x/k_1k_2)^{1/2}}\sum_{(\log x)^5<d}\sum_{s\ll x^{1/2}/(d^2k_1^{1/2}k_2^{1/2})}1\ll x(\log x)^{-3}.
\end{equation*}
Finally the contribution of the terms with $d\mid (k_1,s)$ is less than 
\begin{equation*}
\sum_{d>(\log x)^5}\sum_{dk_1k_2\le y_1}\sum_{\max \{r, ds\}\ll (x/dk_1k_2)^{1/2}}1\ll x(\log x)^{-3}.
\end{equation*}
It remains to evaluate the contribution to $S_4(x)$ of the terms where $k_2>(\log x)^5$.
Since $k_2=(k_1k_2,r^\infty)$, we have $q(k_2)\mid r$ where $q(k_2)=\prod_{p\mid k_2}p$ is the squarefree kernel of $k_2$.
Thus,   the contribution of $r$ is bounded by $x^{1/2}/(q(k_2)(k_1k_2)^{1/2})$, and then the corresponding summation of all the $k_1,k_2,r,s$ with $k_2>(\log x)^5$ is less than
\begin{equation*}
\sum_{(\log x)^5<k_2<y_1}\sum_{k_2k_1<y_1}\frac{x}{k_1k_2q(k_2)}\ll x(\log x)^{-3}\sum_{k_2\ge 1}\frac{1}{q(k_2)k_2^{1/5}}\ll x(\log x)^{-3}.
\end{equation*}

The rest of this section is devoted to the sum $S_5(x)$, which can be written as
\begin{equation*}
S_5(x)=\sum_{\substack{{k_1k_2\le y_1}\\{(k_1,k_2)=1}\\{k_2\le (\log x)^{5}}}}\sum_{\substack{{r^2+s^2\le x/(k_1k_2)}\\{r,s>0, (r,s)=1}\\{q(k_2)\mid r}\\{(k_1,r)=1}\\{(k_1s)^\sharp\le (\log x)^{45}}}}K(r)W(r).
\end{equation*} 
If we replace $r$ by $q(k_2)r'$, the sum over $r'$ has the shape 
\begin{equation*}
S_R=\sum_{r'<R,\, (r',s)=1}K(r'q(k_2))W(r'q(k_2))
\end{equation*}
for some quantity $R=R(s,k_1,k_2)\ll x^{1/2}/(q(k_2)\sqrt{k_1k_2})$.
It is then standard to complete the sum:
\begin{equation*}
S_R=\frac{1}{k_1k_2s}\sum_{h=1}^{k_1k_2s} \bigg( \sum_{\substack{{a=1}\\{(a,s)=1}}}^{k_1k_2s}K(aq(k_2))W(aq(k_2))\e\Big (\frac{ah}{k_1k_2s}\Big) \bigg) \bigg( \sum_{r'<R}\e\Big ({-}\frac{hr'}{k_1k_2s}\Big ) \bigg).
\end{equation*}
As before, the inner sum over $r'$ is geometric and is $\ll \min \big \{ R, \| {h}/{k_1k_2s} \| ^{-1}\big \}$.
Let $S_a$ denote the inner sum over the variable $a$, which is a complete sum. Applying the Chinese remainder theorem many times, we have:
\begin{equation*}
S_a=\prod_{p\mid (k_1s)^\flat}\Big (\sum_{a\Mod p} K_p (a)\e (ha/p)\Big )\prod_{p^\nu\| k_2(k_1s)^\sharp}\Big (\sum_{a\Mod{p^\nu}}W_{p^\nu}(a)\e (ha/p^\nu)\Big ),
\end{equation*}
where
\begin{equation*}
K_p(a)=\e\Big ( \frac{-aq(k_2)\overline{k_2(a^2q(k_2)^2k_2^2(k_1s/p)^2+s^2)}}{p}\Big ).
\end{equation*}
and $W_{p\nu}$ is an exponential term modulo $p^\nu$ whose argument is a similar rational function in~$a$.
Since $k_2(k_1s)^\sharp\le (\log x)^{50}$, a trivial bound for the sums on the $a\Mod{p^\nu}$ when $p^\nu\mid k_2(k_1s)^\sharp$ is sufficient, yielding
\begin{align*}
S_a &\ll \prod_{p\mid (k_1s)^\flat}\Big|\sum_{a\Mod p} K_p (a)\e (ha/p)\Big| \prod_{p^\nu\| k_2(k_1s)^\sharp} p^\nu \\
&\ll (\log x)^{50} \prod_{p\mid (k_1s)^\flat}\Big|\sum_{a\Mod p} K_p (a)\e (ha/p)\Big|.
\end{align*}
(It is in fact possible to find in~\cite[Appendix B]{WX} a useful nontrivial bound for the sums that we have estimated trivially.)

Since $(k_1s)^\flat$  is squarefree, we can apply
Weil's bound for exponential sums of a rational function. The formulation we use is a particular case derived from~\cite[equation (3.5.2)]{De77}.

\begin{lemma}\label{Deligne}
Let $\PP^1(\F_p)$ be the projective line on $\F_p$, and let $f\colon \PP^1(\F_p)\rightarrow \PP^1(\F_p)$ be a nonconstant rational function.
For all $u\in\PP^1(\F_p)$, let $v_u(f)$ be the order of the pole of $f$ at $u$ if $f(u)=\infty$ and $v_u (f)=0$ otherwise.
Then we have 
\begin{equation}\label{v}
\Big |\sum_{\substack{{u\in\PP^1(\F_p)}\\{f(u)\not =\infty}}}\e \Big (\frac{f(u)}{p}\Big )\Big |\le \sum_{v_u(f)\not =0}(1+v_u(f))p^{1/2}
\end{equation}
\end{lemma} 

Since $K_p(a)e(ha/p)$ has at most $3$ poles (including the pole at~$\infty$), which are simple, we have 
\begin{equation*}
\Big |\sum_{a\Mod p} K_p (a)\e (ha/p)\Big |\le 6\sqrt{p},
\end{equation*}
from which we deduce that 
\begin{equation*}
|S_a|\le 6^{\omega ((k_1s)^\flat)}\sqrt{k_1s}(\log x)^{50}.
\end{equation*}

Returning to $S_R$, we have obtained
\begin{equation*}
S_R\ll \frac{R}{k_1k_2s} 6^{\omega ((k_1s)^\flat)}\sqrt{k_1s}(\log x)^{50}+ 6^{\omega ((k_1s)^\flat)}\sqrt{k_1s}(\log x)^{51},
\end{equation*}
which gives the following upper bound for $S_5(x)$:
\begin{equation*}\begin{split}
S_5(x)&\ll (\log x)^{51}\sum_{\substack{{k_1k_2\le y_1}\\{k_2\le (\log x)^5}}}\sum_{s\ll \sqrt{x/(k_1k_2)}} 6^{\omega ((k_1s)^\flat)}\Big (\frac{\sqrt{x}}{q(k_2)\sqrt{k_1k_2s}}+\sqrt{k_1s}\Big )\\
& \ll x^{3/4}y_1^{3/4}(\log x)^{63}.\\
\end{split}
\end{equation*}
If we take $y_1=x^{1/3}(\log x)^{-100}$ we obtain $S_5\ll x(\log x)^{-12}$, which is enough for the proof of Theorem \ref{1+2}.

\section{Three linear factors (Theorem~\ref{1+1+1})} \label{1+1+1 sec}

In this section we consider one of the simplest cases of a product of three linear factors, namely the case $f(n)=n(n+1)(2n+1)$.
Since $(n,(n+1)(2n+1))=(n+1, 2n+1)=1$, our exponential sum is now
\begin{align*}
S(f,x)&=\sum_{n\le x}\sum_{\substack{{r\Mod n}\\{f(r)\equiv 0\Mod n}}}\e\Big (\frac{r}{n}\Big ) \\
&=\sum_{n_1n_2n_3\le x}\sum_{\substack{{r\Mod {n_1n_2n_3}}\\{n_1\mid r}\\{r+1\equiv 0\Mod {n_2}}\\{2r+1\equiv 0\Mod {n_3}}}}\e\Big (\frac{r}{n_1n_2n_3}\Big );
\end{align*}
note that the inner sum has one term when $n_1$, $n_2$, and $n_3$ are pairwise coprime and $n_3$ is odd, and no terms otherwise.

Let $y_2=x^{1/3}(\log x)^B$ with $B>0$ to be specified.
As in the previous sections we split the sum $S(f,x)$, writing
$
S(f,x)=\sum_{i=1}^4S_i(x)
$
where
\begin{align*}
S_1(x) &= \sum_{\substack{n_1n_2n_3\le x \\ n_1 > y_2}} \sum_{\substack{{r\Mod {n_1n_2n_3}}\\{n_1\mid r}\\{r+1\equiv 0\Mod {n_2}}\\{2r+1\equiv 0\Mod {n_3}}}}\e\Big (\frac{r}{n_1n_2n_3}\Big ) \\
S_2(x) &= \sum_{\substack{n_1n_2n_3\le x \\ n_1 \le y_2 \\ n_2 > y_2}} \sum_{\substack{{r\Mod {n_1n_2n_3}}\\{n_1\mid r}\\{r+1\equiv 0\Mod {n_2}}\\{2r+1\equiv 0\Mod {n_3}}}}\e\Big (\frac{r}{n_1n_2n_3}\Big )\\
S_3(x) &= \sum_{\substack{n_1n_2n_3\le x \\ n_1,n_2 \le y_2 \\ n_3 > y_2}} \sum_{\substack{{r\Mod {n_1n_2n_3}}\\{n_1\mid r}\\{r+1\equiv 0\Mod {n_2}}\\{2r+1\equiv 0\Mod {n_3}}}}\e\Big (\frac{r}{n_1n_2n_3}\Big ) \\
S_4(x) &= \sum_{\substack{n_1n_2n_3\le x \\ n_1, n_2, n_3 \le y_2}} \sum_{\substack{{r\Mod {n_1n_2n_3}}\\{n_1\mid r}\\{r+1\equiv 0\Mod {n_2}}\\{2r+1\equiv 0\Mod {n_3}}}}\e\Big (\frac{r}{n_1n_2n_3}\Big ).
\end{align*}

\subsection{The first three sums}

Using a method similar to Section 2 above, the solution $r$ of the congruences in the above sums can be written as 
\begin{equation*}
r=n_1\big({-}\overline{(n_1n_3)}_{n_2}n_3-\overline{(2n_1n_2)}_{n_3}n_2\big),
\end{equation*}
and therefore the exponential summand in the $S_i(x)$ becomes
\begin{equation}\label{expr}
\e\Big (\frac{r}{n_1n_2n_3}\Big )=\e\Big (\frac{-\overline{n_1n_3}}{n_2}-\frac{\overline{2n_1n_2}}{n_3}\Big ).
\end{equation}

For $S_1(x)$ this gives:
\begin{equation*}
S_1(x)=\sum_{\substack{{n_2n_3\le x/y_2}\\{(2n_2,n_3)=1}}}\sum_{y_2<n_1\le x/(n_2n_3)}\e\Big (\frac{\bar n_1 (n_3\overline{(n_3)}_{n_2}-n_2\overline{(2n_2)}_{n_3})}{n_2n_3}\Big ).
\end{equation*}
We apply equation~\eqref{IncSum eqn} with $t=n_3\overline{(n_3)}_{n_2}-n_2\overline{(2n_2)}_{n_3}$. In this case $(t,n_2n_3)=1$ and we obtain:
\begin{equation*}
S_1(x)=\sum_{\substack{{n_2n_3\le x/y_2}\\{(2n_2,n_3)=1}}}\bigg (\Big (\frac{x/(n_2n_3)-y_2}{n_2n_3}\Big )\mu (n_2n_3)+O(\sqrt{n_2n_3} \cdot \tau (n_2n_3)\log x)\bigg ).
\end{equation*}
The error term is $O\big((x/y_2)^{3/2}(\log x)^5\big)$ which is sufficiently small if $B$ is large enough, and therefore
\begin{equation*}\begin{split}
S_1(x)&=x\sum_{(2n_2,n_3)=1}\frac{\mu (n_2)\mu (n_3)}{n_2^2n_3^2}+O (y_2x^\varepsilon) + O\big((x/y_2)^{3/2}(\log x)^5\big) \\
&=x\frac{6}{\pi ^2}\sum_{n_2}\frac{\mu (n_2)}{n_2^2}\prod_{p|2n_2}\Big (1-\frac{1}{p^2}\Big )^{-1}+O\big((x/y_2)^{3/2}(\log x)^5\big)\\
&=x\frac{6}{\pi ^2}\prod_{p\ge 3}\Big (1-\frac{1}{p^2-1}\Big )+O\big((x/y_2)^{3/2}(\log x)^5\big).
\end{split}
\end{equation*}

We handle the sum $S_2(x)$ in the same way, but this time summing first over $n_2$ instead of~$n_1$.
Applying the inversion formula~\eqref{inversion formula},
we can rewrite equation~\eqref{expr} in the following way:
\begin{equation*}
\begin{split}
\e\Big (\frac{r}{n_1n_2n_3}\Big )&=\e\Big (\frac{\bar n_2}{n_1n_3}-\frac{1}{n_1n_2n_3}-\frac{\overline{2n_1n_2}}{n_3}\Big )\\
&=\e\Big (\frac{\bar n_2(1-\overline{(2n_1)}_{n_3}n_1)}{n_1n_3}\Big )+O\Big (\frac{1}{n_1n_2n_3}\Big ).
\end{split}
\end{equation*}
The error term $O(1/n_1n_2n_3)$ yields a contribution to $S_2(x)$ that is less than $O((\log x)^3)$.
Then we apply equation~\eqref{IncSum eqn}:
\begin{equation*}
S_2(x)=\sum_{\substack{{n_1n_3\le x/y_2}\\{n_1\le y_2}\\{(2n_1,n_3)=1}}}\bigg (\Big (\frac{x}{n_1n_3}-y_2\Big )\frac{\mu (n_1n_3)}{n_1n_3}+O(\sqrt{n_1n_3}\cdot \tau (n_1n_3)\log x) \bigg).
\end{equation*}
We finish in the same way as for~$S_1(x)$, obtaining the same asymptotic formula.

For $S_3(x)$ the corresponding method is to write 
\begin{equation*}
\e\Big (\frac{r}{n_1n_2n_3}\Big )=\e \Big ( \frac{\bar n_3(1-\overline{(n_1)}_{n_2}2n_1)}{2n_1n_2}\Big )+O\Big (\frac{1}{n_1n_2n_3}\Big ),\
\end{equation*}
\bibliographystyle{plain}
and then after applying equation~\eqref{IncSum eqn}
\begin{equation*}
S_3(x)=\sum_{\substack{{n_1n_2\le x/y_2}\\{\max \{n_1,n_2\}\le y_2}\\{(n_1,n_2)=1}}}\Big (\Big (\frac{x}{n_1n_2}-y_2\Big )\frac{\mu (2n_1n_2)}{2n_1n_2}+O(\sqrt{n_1n_2}\cdot \tau (n_1n_2)\log x)\Big ).
\end{equation*}
The corresponding main term this time is
\begin{equation*}
S_3(x)=x\sum_{(n_1,n_2)=1}\frac{\mu (2n_1n_2)}{2n_1^2n_2^2}=-\frac{4x}{\pi ^2}\prod_{p\ge 3}\Big (1-\frac{1}{p^2-1}\Big )+O\big((x/y_2)^{3/2}(\log x)^5\big).
\end{equation*}
Summing these contributions of $S_1(x)$, $S_2(x)$, and $S_3(x)$ in the decomposition at the start of this section, we deduce that
\begin{align*}
S(f,x) &= \frac{8x}{\pi ^2}\prod_{p\ge 3}\Big (1-\frac{1}{p^2-1}\Big )+O\big( |S_4(x)| + (x/y_2)^{3/2}(\log x)^5\big) \\
&= x\prod_{p\ge 3}\Big (1-\frac{2}{p^2}\Big )+O\big( |S_4(x)| + x/(\log x)^{3B/2-5} \big).
\end{align*}

\subsection{The sum $S_4(x)$}\label{1+1+1S4}

It remains to handle $S_4 (x)$. Let $y_1=x^{1/3}(\log x)^{-A}$ with $A>0$ to be specified. Let $I$ denote the interval $I=[y_1,y_2]$.
First we remark that the number of summands for which $\min \{n_1,n_2,n_3\}<y_1$ is $\ll y_1y_2^2$, and hence
\begin{equation*}
S_4(x)=\sum_{\substack{{n_1n_2n_3\le x}\\{n_1,n_2,n_3\in I}}}\sum_{\substack{{r\Mod {n_1n_2n_3}}\\{n_1\mid r}\\{r+1\equiv 0\Mod{n_2}}\\{2r+1\equiv 0\Mod{n_3}}}}\e\Big (\frac{r}{n_1n_2n_3}\Big )+O\Big (
\frac{x}{(\log x)^{A-2B}}\Big ).
\end{equation*}
 We introduce a new parameter $z=\exp (\log x/(10\log\log x))$. We now write $n_2=a_2b_2$, $n_3=a_3b_3$ with $P^+ (a_2a_3)\le z<P^-(b_2b_3)$ where $P^+ (n)$ and $P^-(n)$ are, respectively, the largest and smallest prime factors of~$n$.
Using two more parameters $v$ and $w$, we split $S_4(x)$ as 
 $S_4(x)=S_5(x)+S_6(x)+S_7(x)+O(x(\log x)^{2B-A})$,
 where $\max(a_2,a_3)\le v$ in $S_5(x)$, $\max(a_2,a_3) >w$ in $S_6(x)$, and finally $v<\max\{a_2,a_3\}\le w$ in $S_7 (x)$.
 
In $S_5(x)$, since $a_2,a_3$ are small we have:
 \begin{equation}\label{S5}\begin{split}
 S_5(x)&\le\sum_{a_2, a_3\le v}\sum_{n_1,a_2b_2\in I}\sum_{b_3\le x/(n_1n_2a_3)}1\\
 &\ll \frac{x}{\log z}\sum_{a_2, a_3\le v}\sum_{n_1,a_2b_2\in I}\frac{1}{n_1n_2a_3} \ll \frac{x(\log v)^2 (\log\log x)^2}{(\log z)^2}.\\
 \end{split}
 \end{equation}
 The $(\log z)^2$ above comes from the sieving conditions on $b_2$ and $b_3$. In particular, we have used the following inequality 
 \begin{equation}\label{invSifted}
 \sum_{y_1/a_2\le b_2\le y_2/a_2}\frac{1}{b_2}\ll \frac{\log\log x}{\log z},
 \end{equation}
 which can be derived by partial summation from~\cite[Proposition 1]{BBDT12}.
 The bound \eqref{S5} is sufficiently small when $(\log v)(\log\log x )=o(\log z)$ (we will eventually choose $v$ to be a power of $\log x$).
 We remark that this step is the main obstacle to having an upper bound less than $x/(\log x)^2$ in the error term in Theorem~\ref{1+1+1}.

For $S_6(x)$, we estimate each summand trivially by~$1$; therefore we may assume that $a_3>w$ is abnormally large (the bound when $a_2>w$ is exactly the same) and that $n_2$ is unrestricted. Following some ideas of Hooley~\cite{H67}, we note that if $a_3>w$, then either $\omega (a_3)\ge \log w/ (2\log z)$, or else there exists $d>w^{1/4}$ such that $d^2 \mid  a_3$. We therefore have (ignoring here the condition that $b_3$ has no small prime factors)
 \begin{equation}\begin{split}\label{S6}
 S_6(x)&\le\sum_{n_1,n_2\in I}\sum_{w^{1/4}<d<y_2}\sum_{b_3\le x/(n_1n_2d^2)}1+ \sum_{n_1,n_2\in I}\sum_{\substack{{a_3<y_2}\\{\omega (a_3)\ge (\log w)/2\log z}}}\sum_{b_3\le x/ (n_1n_2a_3)}1\\
 &\ll \sum_{n_1,n_2\in I}\sum_{d>w^{1/4}}\frac{x}{n_1n_2d^2}+\sum_{n_1,n_2\in I}\sum_{a_3\le y_2}\frac{x 2^{\omega (a_3)- \log w / (2\log z)}}{n_1n_2a_3}\\
 &\ll \frac{x (\log\log x)^2}{w^{1/4}}+ x2^{-\log w/(2\log z)}(\log\log x)^2(\log x)^2.\\
 \end{split}
 \end{equation}
 
 It remains to handle the term $S_7 (x)$. We begin in the same way as for $S_1 (x)$:
 \begin{equation}  \label{S7 start}
 S_7 (x)=\sum_{\substack{a_2b_2,a_3b_3\in I \\ P^+ (a_2a_3)\le z<P^-(b_2b_3) \\ v<\max(a_2,a_3)\le w}} \sum_{n_1\le x/(n_2a_3b_3)}\e\Big (\frac{\bar n_1 (a_3b_3\overline{(a_3b_3)}_{n_2}-n_2\overline{(2n_2)}_{a_3b_3})}{n_2a_3b_3}\Big ).
 \end{equation}
Unfortunately, equation~\eqref{IncSum eqn} is not sufficient here. Since $a_3$ is not too small and  
 $b_3$ is not too big, the denominator has three factors not too small and we can apply the recent work of Wu and Xi  \cite{WX} on the $q$-analog of the van der Corput method.
 Such an approach was initiated by Heath--Brown \cite{HB78} and  developed by
 Graham and Ringrose \cite{GR89}, and more recently by  Irving~\cite{Ir14,Ir15} and by Wu and Xi~\cite{WX}, where the arithmetic exponent pairs are obtained when the denominator has
 good factorization properties.
As in the proof of Theorem \ref{1+2}, we denote by $n^\sharp$ and $n^\flat$ the squarefull and squarefree parts of the integer $n$; and we write 
\begin{equation*}
\e\Big (\frac{\bar n_1 (a_3b_3\overline{(a_3b_3)}_{n_2}-n_2\overline{(2n_2)}_{a_3b_3})}{n_2a_3b_3}\Big )=K(n_1)W(n_1),
\end{equation*}
 with 
 \begin{equation*}\begin{split}
 K(n_1)&=\e\Big (\frac{\bar n_1\overline{(n_2a_3b_3)^\sharp} [a_3b_3\overline{(a_3b_3)}_{n_2}-n_2\overline{(2n_2)}_{a_3b_3}]}{(n_2a_3b_3)^\flat}\Big ),\\
 W(n_1)&=\e\Big (\frac{\bar n_1\overline{(n_2a_3b_3)^\flat} [a_3b_3\overline{(a_3b_3)}_{n_2}-n_2\overline{(2n_2)}_{a_3b_3}]}{(n_2a_3b_3)^\sharp}\Big ).\\
 \end{split}
 \end{equation*}
 
\begin{lemma}\label{exp}
Uniformly for any integers $\alpha, A,N,\delta\in\N$,  $q=q_1q_2q_3$   squarefree integer such that $(\alpha\delta ,q)=1$ and any rational function $R$ with integer coefficients,
we have
\begin{equation*}\begin{split}
\sum_{\substack{{A<n\le A+N}\\{(n,q)=1}}}\e \Big (\frac{\alpha\bar n}{q}+\frac{R(n)}{\delta}\Big )&\ll N^{1/2}q_3^{1/2}+N^{3/4}q_2^{1/4}\\
&\qquad{}+N^{3/4}q_1^{1/8}\delta ^{1/4}3^{\omega (q_1)}
\Big ( \frac{N}{q_1}+\log (q_1\delta )\Big )^{1/4}.\\
\end{split}
\end{equation*}
\end{lemma}
\noindent
We emphasize that the implicit constant above is absolute, and in particular does not depend on~$R$: we handle the contribution of this term $R(n)/\delta$ quite trivially. In our application, the denominator $\delta$ is small, so the prospects for cancellation are modest in any case.

We prove Lemma~\ref{exp} in the next section; assuming the lemma for the moment, we can complete the proof of Theorem~\ref{1+1+1}.
 We apply Lemma~\ref{exp} to the inner sum in equation~\eqref{S7 start} with $q=n_2^\flat b_3^\flat a_3^\flat$ 
 and $\delta =n_2^\sharp b_3^\sharp a_3^\sharp$, where $N=x/(n_2a_3b_3)$. After doing so, by positivity we may again assume that $v<a_3\le w$ (the bound when $v<a_2\le w$ is exactly the same) and ignore all restrictions upon~$n_2$. We obtain 
 \begin{multline*}
 S_7 (x)\ll \sum_{n_2\in I}\sum_{\substack{{a_3b_3\in I}\\{v<a_3\le w}}}\Big\{ \sqrt{\frac{x}{n_2a_3b_3}}\sqrt{a_3^\flat}+\Big (\frac{x}{n_2a_3b_3}\Big )^{3/4}(b_3^\flat )^{1/4}\\
+\Big (\frac{x}{n_2a_3b_3}\Big )^{3/4}(n_2^\flat )^{1/8}(n_2^\sharp a_3^\sharp b_3^\sharp)^{1/4}3^{\omega (n_2^\flat )}\Big (\frac{x}{n_2n_2^\flat a_3b_3}+\log x\Big )^{1/4}\Big\}.
 \end{multline*} 
 We now have to compute all the different sums:
 \begin{equation}\label{S7}
 S_7(x)\ll \sqrt{x}y_2\sqrt{w}(\log z)^{-1}+ (xy_2)^{3/4}v^{-1/4}+x^{23/24 +\varepsilon}.
 \end{equation}
 Theorem~\ref{1+1+1} now follows from the estimates~\eqref{S5}, \eqref{S6}, and~\eqref{S7} upon taking $w=x^{1/24}$, $v=(\log x)^{4B}$, $B=10$, and $A=30$, for example.

\section{Short exponential sums} \label{ses sec}
\subsection{Proof of Lemma~\ref{exp}} \label{ses1 sec}
 In this section we prove Lemma \ref{exp}, which will complete the proof of Theorem~\ref{1+1+1}.
 This lemma is in fact a variant of a particular case of a result of Wu and Xi~\cite[Theorem 3.1 and Proposition 3.2]{WX}.
While we do not need to introduce significant new ideas, the results of \cite{WX} cannot be applied directly in our context
 because we need a more precise version of the function $N^\varepsilon$ in our error bounds. 
Careful attention to their paper reveals that it is possible to adapt some arguments to replace
this $N^\varepsilon$ by a quantity of the type $C^{\omega (q_1)} (\log N)^\alpha$.
In many circumstances such a refinement is not necessary, but for us it is important due to the very restricted range of the factor~$a_3$.
 
 For brevity we will write $J={}]A,A+N]$, $W(n)=\e (R(n)/\delta)$, and
 \begin{equation*}
 E(J)=\sum_{\substack{{n\in J}}}\e \Big (\frac{\alpha  \bar n }{q}\Big )W(n)
 \end{equation*}
 for the sum to be estimated.
We begin by remarking that we may assume that $q_2<N$ and $q_3<N$, for otherwise the lemma is trivial.
For any function $\Psi$ and any $h\in\Z$ we define 
\begin{equation*}
\Delta_h (\Psi)(n)=\Psi (x)\overline{\Psi (x+h)}.
\end{equation*}

 \begin{lemma}\label{A-process}
 Let $q=q_1q_2$ with $(q_1,q_2)=1$, $J={}]A,A+N]$ an interval and $\Psi _i\colon \Z/q_i\Z\rightarrow \C$.
 Then for $1\le L\le N/q_2$, we have
 \begin{equation*}
 \Big |\sum_{n\in J}\Psi_1 (n)\Psi _2 (n)\Big |^2\ll \| \Psi _2\| _\infty\Big (  L^{-1}N^2+L^{-1}N\sum_{0<|\ell |\le L}\Big |\sum_{\substack{{n\in J}\\{n+\ell q_2\in J}}}
 \Delta_{\ell q_2}(\Psi _1)(n)\Big |\Big ).
 \end{equation*}
 \end{lemma}
 \noindent
 This formula, which Wu and Xi call an $A$-process by analogy with the $A$-process of the classical van der Corput method, follows from the proof of~\cite[Lemma~3.1]{WX}. 
 
 We apply this lemma with 
 \begin{equation*}
 \psi _1 (n)=\e\Big (\frac{\alpha\bar q_3\bar n}{q_1q_2}\Big )W(n),
 \quad
 \psi_2 (n)=\e\Big (\frac{\alpha\overline{q_1q_2}\bar n}{q_3}\Big ).
 \end{equation*}
 Writing $L_3=[N/q_3]$ (which is at least~$1$), this gives (see also the beginning of the proof of~\cite[Theorem~3.1]{WX})
 \begin{equation} \label{E2}
 |E(J)|^2\ll L_3^{-1}N^2+L_3^{-1}N\sum_{0<|\ell _3|\le L_3}|U(\ell _3)|,
 \end{equation}
 with 
 \begin{equation*}
 U(\ell _3 )=\sum_{n\in J(\ell _3)}\e\Big (\frac{\alpha \bar q_3(\bar n -\overline{(n+\ell _3q_3)})}{q_1q_2}\Big )W(n)\overline{W(n+\ell _3q_3)}
 \end{equation*}
 where $J(\ell _3)$ is some interval contained in~$J$. Note that we may write $U(\ell_3) = \sum_{n\in J} \Psi_3(n)\Psi_4(n)$ where
\begin{equation} \label{psi3 psi4}
\begin{split}
\psi_3(n) &= \e\Big (\frac{\alpha \bar q_2\bar q_3(\bar n -\overline{(n+\ell _3q_3)})}{q_1}\Big )W(n)\overline{W(n+\ell _3q_3)} {\bf 1}_{J(\ell _3)}(n), \\
\psi_4(n) &= \e\Big (\frac{\alpha \bar q_1\bar q_3(\bar n -\overline{(n+\ell _3q_3)})}{q_2}\Big ),
\end{split}
\end{equation}
where ${\bf 1}_{J(\ell _3)}$ is the indicator function of $J(\ell _3)$.
We again apply  Lemma \ref{A-process} to each $U(\ell _3)$, writing $L_2=[N/q_2]$ (which again is at least~$1$); we have written $\psi_3(n)$ as in equation~\eqref{psi3 psi4} so as to make this choice of $L_2$ valid even when the length of $J(\ell _3)$ is much smaller than~$N$.
 We obtain  
 \begin{equation} \label{U2}
 |U(\ell _3)|^2\ll L_2^{-1}N^2+L_2^{-1}N\sum_{0<|\ell _2|\le L_2}|U(\ell _2,\ell _3)|,
 \end{equation}
 with now 
 \begin{equation*}
 U(\ell _2,\ell _3)=\sum_{n\in J(\ell _2,\ell _3)}\e \Big (\frac{F(n)}{q_1}\Big )\widetilde W(n),
 \end{equation*}
 where $J(\ell_2,\ell _3)$ is some interval contained in~$J(\ell _3)$, and 
\begin{align}\label{defF}
F(n) &= \alpha\overline{q_2q_3} \big[ \bar n-\overline{(n+\ell _3q_3)}-\overline{(n+\ell _2q_2)}+\overline{(n+\ell _2q_2+\ell _3q_3)} \,\big], \\
\widetilde W(n) &= W(n)\overline{W(n+\ell_3q_3)}\overline{W(n+\ell _2q_2)}W(n+\ell _2q_2+\ell _3q_3). \notag
\end{align}
 Then we complete the above sum over $n\in J(\ell _2,\ell _3)$:
 \begin{equation*}\begin{split}
 U(\ell _2,\ell _3)&=\frac{1}{q_1\delta}\sum_{a=1}^{q_1\delta}\e \Big (\frac{F(a)}{q_1}\Big )\widetilde W(a)\sum_{h=1}^{q_1\delta}\sum_{n\in J(\ell_2,\ell _3)}\e\Big (\frac{h(a-n)}{q_1\delta}\Big )\\
 &\ll \frac{N}{q_1\delta}\Big |\sum_{a=1}^{q_1\delta}\e \Big (\frac{F(a)}{q_1}\Big )\widetilde W(a)\Big |+\sum_{1\le h<q_1\delta}\frac{1}{h}\sum_{a=1}^{q_1\delta}\e \Big (\frac{F(a)}{q_1}\Big )\widetilde W(a)
 \e \Big (\frac{ha}{q_1\delta}\Big ).
 \end{split}
 \end{equation*}
 We denote by $\Sigma_a (h)$ the inner sum on $a$ in the second term and perform the same manipulations as for the sums $S_a$ in the proof of Theorem~\ref{1+2}, resulting in
 \begin{equation}  \label{lastminute}
 \begin{split}
 \Sigma_a (h)& = \sum_{a=1}^{q_1\delta}\e \Big (\frac{F(a)}{q_1}\Big )\widetilde W(a) \e \Big (\frac{ha}{q_1\delta}\Big ) \\
 &= \sum_{u=1}^\delta \widetilde W(q_1u)\e\Big (\frac{hu}{\delta}\Big )\prod_{p|q_1}\sum_{v=1}^p \e\Big (\frac{F(v\delta q_1/p)+hv}{p}\Big ).
 \end{split}
 \end{equation}
 The function $F$ in equation~\eqref{defF} can be rewritten in the following way, with $\lambda =\alpha \overline{q_2q_3}$:
 \begin{equation*}
 F(n)=\frac{\lambda G(n)}{n(n+\ell _2q_2)(n+\ell _3q_3)(n+\ell _2q_2+\ell _3q_3)},
 \end{equation*}
 where $G(n)$ is a polynomial with constant term $\ell _2q_2\ell _3q_3(\ell _2q_2+\ell _3q_3)$ (the exception being when $p \mid (\ell_2q_2+\ell _3q_3)$, in which case we actually have $F(n)=2\lambda\ell_2q_2\ell _3q_3/(n(n+\ell_2q_2)(n+\ell _3q_3))$).
 If $p \nmid\ell _2\ell _3$, the function $F(v\delta q_1/p)+hv$ of $v$ has at most $5$ poles, each pole being simple (including the pole at $\infty$); this is most clearly seen from the definition~\eqref{defF} of $F(n)$. Then by Lemma \ref{Deligne} we have
 \begin{equation*}
\Big |\sum_{v=1}^p \e\Big (\frac{F(v\delta q_1/p)+hv}{p}\Big )\Big |\le 10\sqrt{p}
\end{equation*}
when $p \nmid\ell _2\ell _3$. We deduce from equation~\eqref{lastminute} that 
\begin{equation*}
|\Sigma_a(h)|\le \sum_{u=1}^\delta \prod_{\substack{p\mid q_1 \\ p\nmid\ell_2\ell_3}} 10\sqrt{p} \prod_{\substack{p\mid q_1 \\ p\mid\ell_2\ell_3}} p \le 10^{\omega (q_1)}\delta\sqrt{q_1}(q_1,\ell _2\ell _3)^{1/2}.
\end{equation*}
Then 
\begin{equation} \label{sum U}
\sum_{\ell_2\le L_2}|U(\ell_2,\ell _3)|\ll \delta\sqrt{q_1}\Big (\log q_1+\frac{N}{q_1}\Big )10^{\omega (q_1)}\sum_{\ell _2\le L_2}(q_1,\ell _2\ell_3)^{1/2}.
\end{equation}
For the sum on $\ell_2$ we have for any $\ell_3\le L_3$:
\begin{equation}\label{l2}
 \sum_{\ell _2\le L_2}(q_1,\ell _2\ell_3)^{1/2}\le \sum_{d\mid q_1}\sqrt{d}(\ell _3,q_1/d)^{1/2}\sum_{\ell_2\le L_2/d}1\ll \tau (q_1)L_2(q_1,\ell _3)^{1/2}.
 \end{equation}
Inserting this bound into equation~\eqref{sum U} and tracing the results back through equations~\eqref{U2} and~\eqref{E2} results~in
\begin{equation*}
E(J)
\ll NL_3^{-1/2}+NL_2^{-1/4}+N^{3/4}\Big (\frac{N}{q_1\delta}+\log (q_1\delta)\Big )^{1/4}q_1^{1/8}\delta^{1/4}10^{\omega (q_1)/4}\tau (q_1)^{3/4} .
\end{equation*}
It remains to replace $L_2$ by $[N/q_2]$ and $L_3$ by $[N/q_3]$, and to observe that 
\begin{equation*}
10^{\omega (q_1)/4}\tau (q_1)^{3/4}=80^{\omega(q_1)/4}\le 3^{\omega (q_1)},
\end{equation*}
to finish the proof of  Lemma \ref{exp}.

\subsection{Generalization of Lemma \ref{exp}, the $A^k$-process} \label{ses2 sec}

In this section, we indicate how to iterate the ideas  of the proof of Lemma~\ref{exp}
to obtain bounds for short exponential sums whose denominator can be decomposed into~$k+1$ factors. This generalization relies on techniques from two important papers of Irving~\cite{Ir14,Ir15}; again, our contribution here consists mainly in replacing a factor of $N^\varepsilon$ by a more precise error term.

Essentially, we would like to apply Lemma~\ref{A-process} consecutively $k$ times.
Irving has given a precise formulation of this iteration; to enounce his result, we need to introduce some notation corresponding  to the iterates of the $\Delta_h (\Psi)$ used in the previous section. 
For any complex-valued function $f$, define
\begin{equation*}
f(n;h_1,\ldots ,h_k)=\prod_{S\subset\{ 1,\ldots ,k\}}f\big (n+\sum_{s\in S}h_s\big )^{\sigma (S)},
\end{equation*} 
where $\sigma (S)$ denotes that the complex conjugate is taken when $\#S$ is odd. With this notation, we may quote~\cite[Lemma 2.2]{Ir15}:

\begin{lemma}\label{IrvingkA}
Let $k\in\N$ and $q_0,\dots,q_k\in\N$. For each $0\le i\le k$, let $f_i \colon \Z\rightarrow\C$ be a function with period $q_i$ such that $f_i(n)\ll 1$. Set $q=q_0\cdots q_k$ and $f(n)=\prod_{i=0}^k f_i(n)$. If $I$ is any interval of length at most $N$, then
\begin{multline}\label{iter}
\Big | \sum_{n\in I}f(n)\Big |^{2^k}\ll_k \sum_{j=1}^k N^{{2^k}-2^{k-j}}q_{k-j+1}^{2^{k-j}} \\
+N^{2^k-k-1} \frac q{q_0} \sum_{0<|h_1|\le N/q_1}\cdots\sum_{0<|h_k|\le N/q_k}\Big |\sum_{n\in I(h_1,\ldots ,h_k)}
f_0(n;q_1h_1,\ldots,q_kh_k)\Big |,
\end{multline}
where $I(h_1,\ldots ,h_k)$ is a subinterval of $I$.
\end{lemma}

We also quote the following combinatorial lemma~\cite[Lemma 4.5]{Ir14}:
\begin{lemma}\label{notconstant}
Let $p\ge 3$ be prime and $h_1,\ldots , h_k \in\F_p$. Suppose that for every $b\in\F_p$, the number of subsets $S\subset\{1,\dots,k\}$ with $b = \sum_{s\in S}h_s$ is even. Then some $h_i$ must equal~$0$.
\end{lemma} 

We are now prepared to establish our generalization of Lemma~\ref{exp}.

\begin{lemma}\label{AkB}
Uniformly for any integers $\alpha, A,N$, any positive integers $\delta,k$, any positive integers $q_0,\cdots, q_k$ such that $q=q_0\cdots q_k$ is squarefree and coprime to $\alpha\delta$, and any rational function $R$ with integer coefficients,
\begin{equation*}\begin{split}
\sum_{\substack{{A<n\le A+N}\\{(n,q)=1}}}& \e \Big (\frac{\alpha\bar n}{q}+\frac{R(n)}{\delta}\Big )\ll_k
\sum_{j=1}^k N^{1-1/2^j}q_j^{1/2^j} \\
&+N^{1-1/2^k}q_0^{1/2^{k+1}}\delta^{1/2^k}(2^{k+2}+4)^{\omega(q_0)/2^k}\Big (\Big (\frac{N}{q_0}\Big )^{1/2^k}+(\log q)^{1/2^k}\Big ).
\end{split}
\end{equation*}
\end{lemma}

\begin{proof}
We apply Lemma~\ref{IrvingkA} with 
\begin{equation*}
f(n)=\e \Big (\frac{\alpha\bar n}{q}\Big )W(n)=\e \Big (\frac{\alpha \bar n\overline{(q/q_0)}}{q_0}+
\frac{R(n)}{\delta}\Big )\prod_{i=1}^k\e \Big (\frac{\alpha \bar n\overline{(q/q_i)}}{q_i}\Big ),
\end{equation*}
where as in the previous section $W(n)=\e (R(n)/\delta)$.
Our desire at this point is to apply Weil's bound to the sums over $n$ that arise from equation~\eqref{iter}.
After the same manipulations as in the proof of Lemma \ref{exp}, the analogue of equation~\eqref{lastminute} is now
\begin{equation}\label{lastsecond}\begin{split}
\Sigma_a (m)&=\sum_{u=1}^\delta W(aq_0; q_1h_1,\ldots ,q_kh_k)\e\Big (\frac{mu}{\delta}\Big )\\
&\times\prod_{p|q_0}
\sum_{v=1}^p\e \Big ( \frac{mv+\sum_{S\subset\{1,\ldots,k\}}(-1)^{\#S}\alpha \overline{(v\delta q_0/p+\sum_{s\in S}h_sq_s)}}{p}\Big ).\\
\end{split}
\end{equation}
In order to apply Weil's bound (Lemma~\ref{Deligne}), we need to confirm that the argument of the exponential is nonconstant modulo~$p$ (even if $p\mid m$).

Note that the numerator in equation~\eqref{lastsecond} can be written as
\begin{multline} \label{lastlastsecond}
mv+\sum_{S\subset\{1,\ldots,k\}}(-1)^{\#S}\alpha \overline{(v\delta q_0/p+\sum_{s\in S}h_sq_s)} \\
= mv+\alpha \sum_{t\Mod p} \overline{(v\delta q_0/p+t)} \sum_{\substack{S\subset\{1,\ldots,k\} \\ \sum_{s\in S}h_sq_s \equiv t\Mod p}} (-1)^{\#S}
\end{multline}
When $p\nmid h_1\cdots h_k$ (so that indeed $p\nmid h_1q_1\cdots h_sq_s$), Lemma~\ref{notconstant} implies that at least one of the inner sums on the right-hand side of equation~\eqref{lastlastsecond} has an odd number of terms, and in particular (by considering its parity) is nonzero. In particular, the numerator in equation~\eqref{lastsecond} is nonconstant, and thus Lemma~\ref{Deligne} can be applied.

The number of poles of this numerator is at most $2^k+1$. When $p \nmid h_1\cdots h_k$, Lemma~\ref{Deligne} implies that the corresponding sum is less than 
$2(2^k+1)\sqrt{p}$.
We deduce that $\Sigma_a(m)\ll \delta (2^{k+1}+2)^{\omega (q_0)}\sqrt{q_0}(q_0,h_1\cdots h_k)^{1/2}$.
The proof concludes with quite similar computations to those in the proof of Lemma~\ref{exp}, with the function $\tau (q_1)$ from equation~\eqref{l2} replaced by $\sum_{d_1\cdots d_k \mid q_0}1/\sqrt{d_1\cdots d_k} \ll 2^{\omega (q_0)}$.
\end{proof}

\section{Product of $k$ linear factors} \label{k sec}
This section is devoted to the proof of Theorem \ref{k-linear}. We will skip some details when the arguments are similar to the previous proofs. All implicit constants in this section may depend on~$f$ and~$h$.

Define $A=\prod_{1\le i<j\le k} (a_ib_j-a_jb_i)\prod_{i=1}^ka_i$.
The first step follows the beginning of the proof of Theorem~\ref{1+1} in Section~\ref{1+1General}. However, since $f$ has more than two linear factors, the discussions related
to the greatest common divisor of $A$, $h$, and the denominators $n$ are more delicate. 
This is why in our splitting analogous of \eqref{hacD}, the summation of $\delta$ will be for 
$\delta \mid (hA)^\infty$ instead of $\delta \mid (h, A)^\infty$.
We keep the notations 
$S(f,x,h)=S_>(x,h)+ S_\le (x,h)$, 
where in $S_>(x,h)$ the parameter $\delta \mid (hA)^\infty$ exceeds $B$, but now with $B=(\log x)^{10k}$ instead of $x^{1/5-\varepsilon}$.
We bound $S_>(x,h)$ by $x(\log x)^{k-1}B^{-1+\varepsilon}$ with $\varepsilon >0$ arbitrarily small as in equation~\eqref{>} (indeed $\varepsilon =1/10$ will suffice for us).

For $S_\le (x,h)$, however, an analogous version of equation~\eqref{GHk} is not sufficient.
This is due to the fact that a trivial summation on $\delta <B$ would bring a factor $B>(\log x)^k$
into our error terms, and we can win only a factor $(\log x)^4$ in an error term (denoted by $T_5^{(1)}(r,\delta)$ later in this section) arising from certain denominators $n$ that are the product of several divisors of the same size. 

Nevertheless, we still begin by applying the Chinese remainder theorem as in Section~\ref{<}:
\begin{equation} \label{how many r}
S_\le (f,x,h)=\sum_{\substack{{\delta \mid (Ah)^\infty}\\ {\delta \le B}}}\sum_{\substack{{0\le r <\delta}\\
{f(r)\equiv 0\Mod\delta}}}T(r,\delta),
\end{equation}
with now 
\begin{equation*}
T(r,\delta )=
\sum_{\substack{{n\le x/\delta}\\ {(n,Ah)=1}}}\sum_{\substack{{0\le r_1<n}\\ {f(r_1\delta)\equiv 0\Mod n}}}\e\Big (\frac {hr_1}{n}+\frac{hr\bar n}{\delta}\Big ).
\end{equation*} 
The number of roots $r$ in equation~\eqref{how many r} is $O(A^{\omega (\delta)})$ (see Nagell~\cite[p.~90, Theorem~54]{Na51} for an even more precise result) and thus $O(1)$ with our conventions for implicit constants. Since our polynomial $f$ is the product of $k$ linear functions, generalizations of Lemma~\ref{rg} and equation~\eqref{inversion formula} allow us to write
\begin{equation*}\begin{split}
T(r ,\delta)
& =
\sum_{\substack{{m_1\cdots m_k\le x/\delta}\\ {(m_1\cdots m_k,Ah)=1}}}
\sum_{\substack{{0\le r_1<m_1\cdots m_k}\\ {m_i \mid f_i(r_1\delta) \, (1\le i\le k)}}}\e \Big (\frac{hr_1}{m_1\cdots m_k}+\frac{hr\overline {m_1\cdots m_k}}{\delta}\Big ),
\end{split}
\end{equation*}
where $(m_i,m_j)=1$ for all $1\le i<j\le k$ and
\begin{equation*}
r_1 \equiv -\bar\delta\sum_{i=1}^k \bar a_ib_i \prod_{\substack{{j=1}\\ {j\not =i}\\ } }m_j\overline{(\prod_{\substack{{j=1}\\ {j\not =i}\\ } }m_j})_{m_i}\Mod{m_1\cdots m_k},
\end{equation*}
so that
\begin{equation}\label{r1}
\e \Big (\frac{hr_1}{m_1\cdots m_k}\Big )=\prod_{i=1}^k \e\Big ( \frac{ -h\bar\delta  \bar a_ib_i
\overline{\prod_{\substack{{j=1}\\ {j\not =i}\\ } }m_j}}{m_i}\Big ).
\end{equation}
We set $y=x^{1/3}(\log x)^{10k}$ and write
$T(r ,\delta)=T_1(r ,\delta)+T_2(r ,\delta)$,
where the sum $T_2(\alpha ,\delta)$ contains precisely those summands for which $\max \{m_1,\ldots , m_k\}> y$.

We decompose $T_2(r ,\delta) =
S_1+\cdots +S_k$ into $k$ sums, where each $S_i$ is defined by the conditions $m_i> y$ and $m_j\le y$ for $j<i$.
In each $S_i$, we use the inversion formula~\eqref{inversion formula} and apply Lemma \ref{IncSumBis} as in the proof of Theorem \ref{1+1+1}. 
This yields an asymptotic formula of the shape
$S_i =xC_i (f,h,\delta)+O((x/y)^{3/2}\sqrt{\delta}(\log x)^k)$ for some constants $C_i (f,h,\delta)$ similar to the constants found in Section~\ref{1+1General} in the proof of Theorem~\ref{1+1}.

These contributions comprise the main term in Theorem~\ref{k-linear} when $k=3$ and $k=4$.
It remains to find an upper bound for $T_1(r ,\delta)$. This upper bound will be $o(x)$ only for $k=3$ and $k=4$, which allows for our asymptotic formula in those cases; for $k\ge 5$ it provides only a nontrivial upper bound as stated in Theorem~\ref{k-linear}.

We split the sum $T_1(r ,\delta)$ into $k!$ subsums
according to the ordering of the~$m_i$.
Let $T_1^{(1)}(r ,\delta)$ denote the subsum of $T_1(r ,\delta)$ with the additional condition that $m_1\ge m_2\ge \cdots \ge m_k$; the estimate we find for this subsum will hold for all $k!$ subsums. We would like to deal with $T_1^{(1)}(r ,\delta)$ simply by applying Lemma~\ref{AkB}; however, this lemma is not efficient enough when $m_3$ is close to $m_1$. Consequently, we make one more splitting 
$T_1^{(1)}(r ,\delta)=T_3^{(1)}(r ,\delta)+T_4^{(1)}(r ,\delta)$, where
$T_3^{(1)}(r,\delta)$ consists of the terms for which $m_3\le m_1 (\log x)^{-k4^k}$.

Concentrating first on $T_3 ^{(1)}(r ,\delta)$, we do dyadic splittings in all $k$ variables, writing
$T_3 ^{(1)}(r,\delta)$ as $O(\log^k x)$ subsums $S(M_1,\ldots ,M_k)$ 
that are restricted to $m_i$ satisfying $M_i<m_i\le 2M_i$ for all $1\le i\le k$; here the bounds 
$M_i$ are powers of~$2$, satisfying
$2^kM_1\cdots M_k\le x$, such that $M_1\le y$ while $2M_i\le y$ for $2\le i\le k$. The ordering of the $m_i$ also implies that $M_1\ge M_2\ge \cdots \ge M_k$.

By the inversion formula~\eqref{inversion formula}, we can eliminate the variable $m_1$ in the denominator in equation~\eqref{r1}:
\begin{equation*}
\e\Big ( {-}\frac{hb_1\overline{a_1\delta m_2\cdots m_k}}{m_1}\Big )=
\e\Big ( \frac{hb_1\overline{m_1}}{a_1\delta m_2\cdots m_k}\Big )\e \Big (\frac{hb_1}{a_1\delta m_1\ldots m_k}\Big ),
\end{equation*}
The second exponential term above is $1+O(1/\delta M_1\cdots M_k) $ and can be replaced by $1$ with an admissible error.
Thus we have, for some $\lambda$ depending on $m_2,\ldots ,m_k,\delta, r, f$ but not on~$m_1$,
\begin{multline*}
S(M_1,\ldots,M_k) \\
=\sum_{\substack{{M_i<m_i\le 2M_i} \, (2\le i\le k) \\ (m_2\cdots m_k,Ah)=1}}
\sum_{\substack{{M_1<m_1\le 2M_1} \\ {\max \{m_2,\cdots ,m_k\}\le m_1\le x/(m_1\cdots m_2)}\\ {(m_1,Ahm_2\cdots m_k)=1}}}\e\Big (\frac{\lambda \overline{m_1}}{a_1\delta m_2\cdots m_k}\Big ) +E,
\end{multline*}
where $E$ is a sufficiently small error term.
Next we separate the squarefree and the squarefull parts of this denominator.
With  the same notation as in the proof of Theorem~\ref{1+1+1}, we obtain a formula of the shape
\begin{equation*}
\e\Big (\frac{\lambda \overline{m_1}}{a_1m_2\cdots m_k}\Big ) =
\e\Big (\frac{\alpha \overline{m_1}}{(a_1m_2\cdots m_k)^\flat}+\frac{R(n)}{\delta (a_1m_2\cdots m_k)^\sharp}
\Big ).
\end{equation*}
Writing $q^\sharp = (a_1m_2\cdots m_k)^\sharp$, we find that Lemma~\ref{AkB} gives the estimate
\begin{multline*}
S(M_1,\ldots ,M_k)\ll x\sum_{j=3}^k(M_j/M_1)^{1/2^{j-2}}\\
+x(\delta q^\sharp)^{1/2^k}(\log x)^{2^k+3}\big( M_2^{-1/2^{k}} +(\log x)^{1/2^k}M_2^{1/2^{k+1}}M_1^{-1/2^k} \big).
\end{multline*}
Since $M_k \le \cdots \le M_3\le M_1(\log x)^{-k4^k}$, this estimate saves enough powers of $\log x$ to compensate for the number of subsums and the summation over~$\delta$; we obtain a bound for $T_3 ^{(1)}(r ,\delta)$ that is admissible for Theorem~\ref{k-linear}.

It remains to handle $T_4^{(1)}(r ,\delta)$, where a direct application of Lemma~\ref{AkB} is not
sufficient. We will proceed as in the sum $S_4 (x)$ in Section~\ref{1+1+1S4}. 
Recalling that $m_1\ge m_2\ge\cdots\ge m_k$, we let $\ell\in\{ 2,\ldots ,k\}$ be the largest index such that $m_\ell> m_1(\log x)^{-k4^k}$.
For $2\le i\le \ell$, we factor $m_i =c_id_i$ with 
$P^+(c_i)\le z <P^-(d_i)$ where $z=\exp (\log x/(800k\log\log x))$. 
We also introduce the parameters $w=x^{1/(50k)}$ and $v=(\log x)^{10k4^k}$.

Let $T_6^{(1)}(r ,\delta)$ be the analogue of the sum $S_6(x)$ in Section~\ref{1+1+1S4}, which is the contribution of the $(m_1,\ldots ,m_k)$ such that $c_i>w$ for some $i\in\{ 2,\ldots ,\ell\}$. With a computation similar to the one in equation~\eqref{S6}, we find:
\begin{equation*}
T_6^{(1)}(r ,\delta)\ll \frac{x(\log x)^k}{w^{1/4}}+x2^{-{\log w}/(2\log z)}(\log x)^k \ll x{(\log x)^{k-8\log2}},
\end{equation*}
where the final estimate is sufficient since $8\log2>5$.

We next examine $T_7^{(1)}(r ,\delta)$, the analogue of $S_7(x)$ in Section~\ref{1+1+1S4}, corresponding to the case where there is at most one $i\in \{ 2,\ldots,\ell\}$ with $c_i\le v$.
Without lost of generality we can suppose that $c_2\le v$ and $c_i\in [v,w ]$ for all $3\le i\le \ell$.
We introduce some dyadic splittings, similar to those in our treatment of $T^{(1)}_3(r,\delta)$, of $T_7^{(1)}(r ,\delta)$ into $O((\log x)^k)$ sums
$S_7(M_1,\ldots ,M_k)$ with  $M_i<c_id_i\le 2M_i$ for $2\le i\le \ell$ and $M_i<m_i\le 2M_i$
for $i\in\{1\}\cup\{\ell+1,\ldots ,k\}$.
We apply Lemma~\ref{AkB} with $k$ replaced by $k+\ell-4$ and with the parameters
\begin{align*}
N&=2M_1,\\
q_0&=c_2d_2, \\
q_1&=d_3, \, q_2=d_4,\, \ldots ,\, q_{\ell -2}=d_\ell, \\
q_{\ell -1}&=m_{\ell +1}, \, q_\ell = m_{\ell+2},\, \ldots ,\, q_{k-2}=m_k, \\
q_{k-1}&=c_3,\, q_k=c_4,\, \dots,\, q_{k+\ell-4}=c_\ell.
\end{align*}
With the notations $q^\sharp=(a_1m_2\cdots m_k)^\sharp$ and $k_1=k+\ell -4$, Lemma 12 
gives
\begin{equation*}
\begin{split}
S_7&(M_1,\ldots ,M_k) \\
&=\mathop{\sum\nolimits^{*\!\!}}_{m_2,\ldots,m_k} \bigg [\sum_{j=1}^{\ell -2}M_1^{1-1/2^j}
d_{j+2}^{1/2^j}+\sum_{j=\ell -1}^{k-2}M_1^{1-1/2^j}m_{j+2}^{1/2^j}+
\sum_{j=k-1}^{k+\ell -4}M_1^{1-1/2^j}c_{j-k+4}^{1/2^j}\\
&\qquad{}+M_1^{1-1/2^{k_1}}M_2^{1/2^{k_1+1}}(\delta q^\sharp)^{1/2^{k_1}}(2^{k_1}+4)
^{\omega (m_2)/2^{k_1}}\bigg (\Big (\frac{M_1}{M_2}\Big )^{1/2^{k_1}}+(\log x)^{1/2^{k_1}} \bigg)
\bigg ],\\
\end{split}
\end{equation*}
where the asterisk on the sum indicates that the $m_j$ and the $c_j,d_j$ satisfy the previously indicated conditions. At this point, computations analogous to those at the end of Section~\ref{1+1+1S4} give a bound for $T_7^{(1)}(r,\delta)$ that is sufficiently small.

Finally, we examine $T_5^{(1)}(r ,\delta)$, the analogue of $S_5(x)$ in Section~\ref{1+1+1S4}, corresponding to the case where $c_i\le v$ for at least two indices $i\in\{ 2,\ldots ,\ell\}$. The estimation of this part of the sum is where we fail to obtain an asymptotic formula for $S(f,x,h)$ when $k\ge 5$.

Without loss of generality we can suppose that $c_2\le v$ and
$c_3\le v$. If we define $I(m_1)=[m_1/(\log x)^{k4^k} , m_1]$, then
\begin{equation*}
T_5^{(1)}(r ,\delta )\ll \sum_{m_1\le y}\sum_{m_5,\ldots ,m_k\le m_1}\sum_{c_2,c_3\le v}
\sum_{\substack{{c_2d_2\in I(m_1)}\\ {c_3d_3\in I(m_1)}}}\sum_{
\substack{{m_4\le x/(\delta m_1c_2d_2c_3d_3m_5\cdots m_k)}\\ {m_4\le m_1}}}1.
\end{equation*}
(In this line, the sum over $m_5,\ldots , m_k$ does not occur if $k=3$ or $k=4$; in fact, the case $k=3$ is handled exactly in the same way as in the estimation of $S_5(x)$ in Section~\ref{1+1+1S4}.)
We use equation~\eqref{invSifted} for the sum over the sifted variables $d_2$, $d_3$, gaining an additional factor of $(\log z)^2$ in the estimate:
\begin{equation}\label{T5}
T_5^{(1)}(r ,\delta)\ll \frac{x(\log\log x)^2(\log x)^{k-3}}{\delta(\log z)^2}\ll x\delta^{-1} (\log\log x)^4(\log x)^{k-5}.
\end{equation}
Since the number of roots $r$ is bounded uniformly in $\delta$ and $\sum_{\delta \mid (Ah)^\infty}\delta^{-1} \ll1$, we see that the bound~\eqref{T5} is sufficiently small. This completes the proof of Theorem~\ref{k-linear}.



\section*{Acknowledgments} 
The authors are very grateful to Zhizhong Huang for all his remarks on a preliminary version of this manuscript. We also express our deep gratitude to the anonymous referee, both for a close and detailed reading of the manuscript which substantially improved the final version, and for suggesting Theorem~\ref{k-linear} and providing the outline of the proof. The second author was supported in part by a Natural Sciences and Engineering Research Council of Canada Discovery Grant.


\bibliographystyle{amsplain}



\begin{dajauthors}
\begin{authorinfo}[cecile] 
  C\'ecile Dartyge\\
  Institut \'Elie Cartan de Lorraine\\
  BP 70239, 54506 Vand\oe uvre-l\`es-Nancy Cedex, France\\
  cecile.dartyge\imageat{}univ-lorraine\imagedot{}fr \\
  \url{https://www.iecl.univ-lorraine.fr/~Cecile.Dartyge/} 
\end{authorinfo}
\begin{authorinfo}[greg]
  Greg Martin\\
  Department of Mathematics\\
  University of British Columbia\\
  Room 121, 1984 Mathematics Road\\
  Vancouver, BC, Canada \ V6T 1Z2\\
  gerg\imageat{}math\imagedot{}ubc\imagedot{}ca \\
  \url{http://www.math.ubc.ca/~gerg/}
\end{authorinfo}
\end{dajauthors}

\end{document}